\documentclass[preprint]{elsarticle}
\usepackage{amsmath,amssymb,hyperref,algorithmicx,algorithm,color,algpseudocode,tikz,soul,natbib,amsthm,subcaption}

\usepackage[final,authormarkup=none]{changes}
\definechangesauthor[color=blue]{Sriram}
\definechangesauthor[color=red]{Sauleh}
\newcommand{\sriAdd}[1]{\added[id=Sriram]{#1}}

\newcommand{\SSDel}[1]{\deleted[id=Sauleh]{#1}}

\def \mytitle {Sensitivity and Covariance in Stochastic Complementarity Problems with an Application to North American Natural Gas Markets}

\usepackage{amsthm}
\newtheorem{lemma}{Lemma}
\newtheorem{proposition}[lemma]{Proposition}
\newtheorem{theorem}[lemma]{Theorem}
\newtheorem{corollary}[lemma]{Corollary}
\theoremstyle{definition}
\newtheorem{definition}[lemma]{Definition}
\newtheorem{assumption}{Assumption}

\numberwithin{equation}{section}

\newcommand{\ie}			{\textit{i.e., }}
\DeclareMathOperator*{\cov}  {\mathcal{C}ov}

\DeclareMathOperator*{\expec}  {\mathbb{E}}

\DeclareMathOperator* {\sgn}{sgn}
\DeclareMathOperator* {\Max}{Maximize}
\DeclareMathOperator* {\Min}{Minimize}

\definecolor{brown}{HTML}{996633}
\definecolor{dgree}{HTML}{44AF2A}
\definecolor{blue}{HTML}{0000FF}

\newcommand{\code}[1]{\texttt{#1}}

\newcommand{\x}       {  {\mathbf{x}}}
\newcommand{\y}       {  {\mathbf{y}}}
\newcommand{\xstar}   {\textcolor{black}  {\x^*}}
\newcommand{\ystar}   {\textcolor{black}  {\y^*}}
\newcommand{\aaaa}    {  {\theta}}
\newcommand{\randa}   {  {\mathbf{\aaaa}}}
\newcommand{\Expa}    {\textcolor{black}  {\hat{\randa}}}
\newcommand{\f}       {\textcolor{black}  {\mathsf{f}}}
\newcommand{\mer}     {\textcolor{black}  {\psi}}
\newcommand{\Mer}     {\textcolor{black}  {\Phi}}
\newcommand{\gradxf}  {\textcolor{black}  {\nabla_{\x}\f}}

\newcommand{\hessxf}  {\textcolor{black}  {\nabla_{\x}^2\f}}
\newcommand{\Jacob}   {\textcolor{black}  {\mathcal{J}}}
\newcommand{\Hstar}   {\textcolor{black}  {\mathcal{H}}}
\newcommand{\C}       {\textcolor{black}  {\mathsf{C}}}

\newcommand{\perturb} {\textcolor{black}  {\Delta\randa}}
\newcommand{\T}       {\textcolor{black}  {\mathcal{T}}}
\newcommand{\cFi}     {\textcolor{black}  {\mathbf{F}}}
\newcommand{\cF}      {\textcolor{black}  {\mathbf{\cFi}}}

\newcommand{\nabPhi}  {\textcolor{black}  {\mathcal {M} }}
\newcommand{\sens}    {\textcolor{black}  {\beta }}

\newcommand{\Lipsh}	  {\textcolor{black} {\mathcal{L}(\xstar;\randa)}}

\newcommand{\randomxi}  {\textcolor{black}  {\omega}}

\renewcommand{\P}     			{				   {\text{P}}}
\renewcommand{\C}			{				   {\text{C}}}
\newcommand{\N}				{				   {\text{N}}}
\newcommand{\A}				{				   {\text{A}}}
\newcommand{\Y}				{				   {\text{Y}}}
\newcommand{\Q}				{				   {\text{Q}}}
\newcommand{\Ano}			{				   {\A_n^+}}
\newcommand{\Ani}			{				   {\A_n^-}}

\newcommand{\QpcyC}			{\textcolor{black} {\Q_{pcny}^{\C}}}
\newcommand{\QpyP}			{\textcolor{black} {\Q_{pny}^{\P}}}
\newcommand{\Qpay}			{\textcolor{black} {\Q_{pay}^{\A}}}
\newcommand{\Qa}			{\textcolor{black} {\Q_{ay}^{\A}}}

\newcommand{\Price}			{				   {\mathbf{\pi}}}
\newcommand{\piC}			{\textcolor{black}  {\mathbf{\Price_{cy}}}}
\newcommand{\piA}			{\textcolor{black}  {\mathbf{\Price^{\A}_{ay}}}}

\newcommand{\X}				{				   {\text{X}}} 
\newcommand{\Xp}[1][y]			{\textcolor{black} {\X^{\P}_{p{#1}}}}
\newcommand{\Xa}[1][y]			{\textcolor{black} {\X^{\A}_{a{#1}}}}
\newcommand{\CapP}			{\textcolor{black} {\text{CAP}^{\P}_{py}}}
\newcommand{\CapA}			{\textcolor{black} {\text{CAP}^{\A}_{ay}}}

\newcommand{\avl}			{\textcolor{black} {\alpha^{\P}}}
\newcommand{\Qpo}			{\textcolor{black} {\ensuremath{\hat{\Q}_{p0}}}}
\newcommand{\Qao}			{\textcolor{black} {\ensuremath{\hat{\Q}_{a0}}}}
\newcommand{\costP}			{\textcolor{black} {\mbox{Gol}}}
\newcommand{\costA}			{\textcolor{black} {\gamma_{ya}^{\A}(\randomxi)}}
\newcommand{\piXP}			{\textcolor{black} {\Price^{\X\P}_{py}(\randomxi)}}
\newcommand{\piXA}			{\textcolor{black} {\Price^{\X\A}_{ay}(\randomxi)}}
\renewcommand{\l}			{\textcolor{black} {l^{\P}_{py}(\randomxi)}}
\newcommand{\q}				{\textcolor{black} {q^{\P}_{py}(\randomxi)}}
\newcommand{\g}				{\textcolor{black} {g^{\P}_{py}(\randomxi)}}
\newcommand{\DemS}			{\textcolor{black} {D^{\C}_{cy}(\randomxi)}}
\newcommand{\DemI}			{\textcolor{black} {E^{\C}_{cy}(\randomxi)}}

\newcommand{\loss}			{\textcolor{black} {L}}
\newcommand{\lossP}			{\textcolor{black} {\loss^{\P}_{py}(\randomxi)}}
\newcommand{\lossA}			{\textcolor{black} {\loss^{\A}_{ay}(\randomxi)}}
\newcommand{\df}			{\textcolor{black} {\mathsf{df}_y(\randomxi) }}

\newcommand{\dab}			{\textcolor{black} {\delta^1_{py}}}
\newcommand{\dac}[1][y]		{\textcolor{black} {\delta^2_{p{#1}}}}
\newcommand{\dad}			{\textcolor{black} {\delta^3_{pny}}}

\newcommand{\dah}			{\textcolor{black} {\delta^5_{ay}}}
\newcommand{\dai}[1][y]		{\textcolor{black} {\delta^6_{a{#1}}}}

\title{\mytitle}
\date{}

\author[civil]{Sriram Sankaranarayanan\corref{cor1}}
\ead{ssankar5@jhu.edu}

\author[jgcri]{Felipe Feijoo}
\ead{felipe.feijoo@pnnl.gov}

\author[civil,math,diw]{Sauleh Siddiqui}
\ead{siddiqui@jhu.edu}

\address[civil]{Department of Civil Engineering, Johns Hopkins University}
\address[jgcri]{Pontificia Universidad Cat\'{o}lica de Valpara\'{i}so, Chile}
\address[math]{Department of Applied Mathematics and Statistics, Johns Hopkins University}
\address[diw]{German Institute for Economic Research (DIW Berlin)}

\cortext[cor1]{Corresponding Author: ssankar5@jhu.edu}

\biboptions{square,sort&compress}
\begin{document}

\begin{abstract}
	We provide an efficient method to approximate the covariance between decision variables {and uncertain parameters} in solutions to a general class of stochastic nonlinear complementarity problems. We also develop a sensitivity metric to quantify uncertainty propagation by determining  the change in the variance of the output due to a change in the variance of an input parameter. The covariance matrix of the solution variables quantifies the uncertainty in the output and pairs correlated variables and parameters. The sensitivity metric helps in identifying the parameters that cause maximum fluctuations in the output. The method developed in this paper optimizes the use of gradients and matrix multiplications which makes it particularly useful for large-scale problems. Having developed this method, we extend the deterministic version of the North American Natural Gas Model (NANGAM), to incorporate effects due to uncertainty in the parameters of the demand function, supply function, infrastructure costs, and investment costs. We then use the sensitivity metrics to identify the parameters that impact the equilibrium the most.
\end{abstract}
\maketitle

\section{Introduction}
Complementarity models arise naturally out of various real life problems. A rigorous survey of their application is available in \cite{Ferris1997}. Authors in \cite{Martin2014,huppmann2014market,Feijoo,oke2016mitigating,Abada2013,christensen2015mixed} use complementarity problems to model markets from a game theoretic perspective \citep{siddiqui2016determining,anderson2004nash}, where the complementarity conditions typically arise between the marginal profit and the quantity produced by the producer. In the field of mechanics, they typically arise in the context of frictional contact problems \citep{kwak1988complementarity}, where there is a complementarity relation between the frictional force between a pair of surfaces and the distance of separation between them. {With a wide range of applications, understanding the characteristics of solutions to complementarity problems becomes important for advancing the field. In this paper, we focus on studying the characteristics of solutions to complementarity problems under uncertainty.} 

The behavior of a solution to a complementarity problem with random parameters was first addressed in \cite{Gurkan1999}, where such problems were referred to as stochastic complementarity problems (SCP). Authors in \cite{Shanbhag2013a,Chen2005expected,Gabriel2009,Egging2016,Jiang2008} define various formulations of SCP for different applications and have devised algorithms to solve the problem. Authors in \cite{Lamm2016} compute confidence intervals for solution of the expected value formulation of the problem, however they do not have efficient methods to find the second-order statistics for large-scale complementarity problems.

Large-scale problems, those with over 10,000 decision variables and uncertain parameters arise naturally out of detailed market models and there is considerable interest in studying, understanding and solving such models. For example, \cite{chen2017dynamic} discuss a case of urban drainage system with large number of variables. \cite{yumashev2017flexible} discuss a case of deciding under large-scale nuclear emergencies. In line with the area of application used in this paper, \cite{Gabriel2001} discuss a case of an energy model with large number of variables and parameters. Naturally, developing methods to solve such large-scale problems gained interest. {Authors in} \cite{kopanos2010mip,doi:10.1287/opre.2015.1413} discuss various tools ranging from mathematical techniques (decomposition based) to computational techniques (parallel processing) for solving large-scale optimization problems. \cite{ohno2016approx} uses an approximate algorithm for a large-scale Markov decision process to optimize production and distribution systems. In this paper, we do not present a new method to solve stochastic complementarity problems, but an efficient algorithm to generate second-order information that is flexible enough to be coupled with any existing algorithm that provides a first-order solution.

The objective of this paper is to efficiently obtain second-order statistical information about solution vectors of large-scale stochastic complementarity problems. This gives us information about variability of the equilibrium obtained by solving a nonlinear complementarity problem (NCP) and the correlation between various variables in the solution. {Authors in }\cite{hyett2007valuing} and \cite{cecchi2003} provide examples in the area of clinical pathways and ecology respectively{,} about the utility of understanding the variance of the solution in addition to the mean. They also show that a knowledge of variance aids better understanding and planning of the system. \cite{Agrawal2012} emphasize the necessity to understand covariance as a whole rather than individual variances by quantifying ``the loss incurred on ignoring correlations'' in a stochastic programming model. 

In addition, we also introduce a sensitivity metric which quantifies the change in uncertainty in the output due to a perturbation in the variance of uncertain input parameters. This helps us to directly compare input parameters by the amount of uncertainty they propagate to the solution.

In attaining the above objectives, the most computationally expensive step  is to solve a  system of linear equations. We choose approximation methods over analytical methods, integration, or Monte Carlo simulation because of the computational hurdle involved while implementing those methods for large-scale problems. The method we describe in this paper achieves the following:
\begin{itemize}
	\item The most expensive step has to be performed just once, irrespective of the covariance of the input parameters. Once the linear system of equations is solved, for each given covariance scenario, we only perform two matrix multiplications.
    \item Approximating the covariance matrix and getting a sensitivity metric can be obtained by solving the {above mentioned}  linear system just once.
\end{itemize}

The methods developed in this paper can also be used for nonlinear optimization problems with linear equality constraints. We prove stronger results on error bounds for special cases of quadratic programming.

Having developed this method, we apply it to a large-scale stochastic natural gas model for North America, an extension of the deterministic model developed {in} \cite{Feijoo} and determine the covariance of the solution variables. We then proceed to identify the parameters which have the greatest impact on the solution. A Python class for efficiently storing and operating on sparse arrays of dimension greater than two is created. This is useful for working with high-dimensional problems which have an inherent sparse structure in the gradients.

We divide the paper as follows. Section \ref{sec:MathForm} formulates the problem and mentions the assumptions used in the paper. It then develops the algorithm used to approximate the solution covariance and provides proofs for bounding the error. Section \ref{sec:total_sensitivity} develops a  framework to quantify the sensitivity of the solution to each of the random variables. Section \ref{sec:application_to_unconstrained_minimization} shows how the result can be applied for certain optimization problems with equality constraints. Having obtained the theoretical results, section \ref{Sec:Duopoly} gives an example of a oligopoly where this method can be applied and compares the computational time of the approximation method with a Monte-Carlo method showing the performance improvement for large-scale problems. Section \ref{sec:NANGAM} describes the Natural Gas Model to which the said method is applied. Section \ref{sec:conclusion_and_future_work} discusses the possible enhancements for the model and its limitations in the current form.

\section{Approximation of covariance} 
\label{sec:MathForm}
For the rest of the paper, all bold quantities are vectors. A subscript $i$ for those quantities refer to the $i$-th component of the vector in Cartesian representation.
\subsection{Definitions} 
We define a complementarity problem and a stochastic complementarity problem which are central to the results obtained in this paper. We use a general definition of complementarity problems and stochastic complementarity problems as stated below.
\begin{definition}
 \citep{facchinei2007finite} Given $\cF:\mathbb{R}^{n\times m}\mapsto \mathbb{R}^n$, and parameters $\randa\in \mathbb{R}^m$, the \emph{parametrized nonlinear complementarity problem} (NCP) is to find $\x\in \mathbb{R}^n$ such that
\begin{align}
\mathbb{K}\ni \x \perp \cF(\x;\randa) \in \mathbb{K}^*\label{eq:NCP}
\end{align}
	where $\mathbb{K}^*$, the dual cone of $\mathbb{K}$ is defined as
	\begin{align}
	\mathbb{K}^* \quad &= \quad \left\{ \x \in \mathbb{R}^n : \mathbf{v}^T\x \geq 0 \quad\forall \mathbf{v}\in \mathbb{K} \right\}
	\end{align}
\end{definition}

\begin{definition}\label{def:SCP}
	Given a cone $\mathbb{K}\in \mathbb{R}^n$ a random function $\cF:\mathbb{K}\times\Omega \mapsto \mathbb{R}^n$, the stochastic complementarity problem (SCP) is to find $x\in \mathbb{R}^n$ such that
	\begin{align}
	 \mathbb{K}\ni \x \perp \expec\cF(\x;\randomxi) \in \mathbb{K}^*\label{eq:SCP}
	\end{align}
\end{definition} 

We assume that we can explicitly evaluate the expectation in \eqref{eq:SCP} using its functional form, and that the SCP can be solved using an existing algorithm.

We now make assumptions on the form of $\mathbb{K}$ in \eqref{eq:NCP}. This form of $\mathbb{K}$ helps in establishing an equivalence between a complementarity problem and a minimization problem which is key to derive the approximation method in this paper. 
\begin{assumption}\label{as:cone}
 $\mathbb{K}$ in \eqref{eq:NCP} is a Cartesian product of half spaces and full spaces, \ie for some $\mathcal{I}\subseteq \left\{ 1,2,\ldots,n \right\}$
\begin{align}
\mathbb{K} \quad &= \quad \left\{ \x\in \mathbb{R}^n: \quad \x_i \geq 0	\quad \mbox{if }i \in \mathcal{I} \right\}
\end{align}
\end{assumption}

We now propose a lemma about the form of the dual cone of $\mathbb{K}$ to understand the special form that it has. This will help us convert the complementarity problem into an unconstrained minimization problem.
\begin{lemma}\label{prop:dual}
	The dual cone $\mathbb{K}^*$ of the set assumed in assumption \ref{as:cone} is
	\begin{align}
	\mathbb{K}^* = \mathbb{K}' &= \left\{ \x \in \mathbb{R}^n:
	\begin{array}{lr}
		\x_i\geq 0 &\mbox{ if }i \in \mathcal{I}\\
		\x_i = 0 &\mbox{ if }i \not\in \mathcal{I}\\
	\end{array}
	 \right\}
	\end{align}
\end{lemma}
\begin{proof}
	Check Appendix A
\end{proof}

\subsection{Preliminaries for approximation} 
{In this subsection, we prove two preliminary results. Firstly, we prove our ability to pose an NCP as an unconstrained minimization problem. Then we prove results on twice continuous differentiability of the objective function, thus enabling us to use the rich literature available for smooth unconstrained minimization. Following that, propositions \ref{prop:uncminNess} and \ref{prop:uncminSuff} help in achieving the former while proposition \ref{prop:diff} along with its corollaries help us in achieving the latter. }
    
    We now define C-functions, which are central to pose the complementarity problem into an unconstrained optimization problem. The equivalent formulation as an unconstrained optimization problem assists us in developing the algorithm.
\begin{definition}\label{def:merit}\citep[Pg. ~72]{facchinei2007finite}
A function $\mer: \mathbb{R}^2\mapsto \mathbb{R}$ is a \emph{C-function} when
\begin{align}
\begin{split}
\mer(x,y)= 0 \qquad &\Leftrightarrow \qquad  x\geq 0 \quad y\geq 0 \quad xy=0
\end{split}\label{eq:CFunction}
\end{align}
We consider the following commonly used C-functions.
\begin{align}
\mer_{FB}(x,y)\quad &= \quad \sqrt{x^2+y^2}-x-y\label{eq:FisBurm}\\
\mer_{min}(x,y) \quad &= \quad \min(x,y)\label{eq:Min}
\end{align}
\end{definition}

Under our assumptions on $\mathbb{K}$, the following two propositions establish the equivalence of the complementarity problem and an unconstrained minimization problem. 

\begin{proposition}
Suppose assumption \ref{as:cone} holds. Then every solution $\xstar(\randa)$ of the parameterized complementarity problem in \eqref{eq:NCP}, is a global minimum of the following function $\f(\x;\randa)$,
	\begin{align}
	\Mer_i(\x,\randa;\cF)\quad &= \quad \left\{ \begin{array}{ll}
	\cF_i(\x,\randa)&\qquad\mbox{if }\quad i\not\in \mathcal{I}\\
	\mer_i(\x_i,\cF_i(\x,\randa))&\qquad\mbox{if }\quad i \in \mathcal{I}
\end{array} \right . \label{eq:exmer1}\\
\f(\x;\randa)\quad &= \quad \frac{1}{2}\|\Mer(\x;\randa;\cF)\|_2^2 \label{eq:exmer2}
	\end{align}
	with an objective value 0, for some set of not necessarily identical C-functions $\mer_i$.\label{prop:uncminNess}
\end{proposition}
\begin{proof}
	Check Appendix A
\end{proof}

\begin{proposition}\label{prop:uncminSuff}
	Suppose assumption \ref{as:cone} holds. If a solution to the problem in \eqref{eq:NCP} exists and $\xstar(\randa)$ is an unconstrained global minimizer of $\f(\x;\randa)$ defined in \eqref{eq:exmer2}, then $\xstar(\randa)$ solves the complementarity problem in \eqref{eq:NCP}.
\end{proposition}
\begin{proof}
	Check Appendix A
\end{proof}

Now given a function $\cF$, and a set $\mathbb{K}$ which satisfies assumption \ref{as:cone}, and a solution of the NCP $\xstar(\Expa)$ for some fixed $\randa = \Expa$, we define a vector valued function $\Mer:\mathbb{R}^{n\times m}\mapsto \mathbb{R}^n$ component-wise as follows.
\begin{align}
\Mer_i(\x,\randa;\cF)\quad &= \quad \left\{ \begin{array}{ll}
	\cF_i(\x,\randa)&\qquad\mbox{if }\quad i\not\in \mathcal{I}\\
	\mer^2(\x_i,\cF_i(\x,\randa))&\qquad\mbox{if }\quad i\in \mathcal{Z}\\
	\mer(\x_i,\cF_i(\x,\randa))&\qquad\mbox{otherwise}
\end{array} \right .\\
\f(\x;\randa)\quad &= \quad \frac{1}{2}\|\Mer(\x;\randa;\cF)\|_2^2 \label{eq:f}\\
\mathcal{Z}\quad &= \quad \left\{ i\in \mathcal{I}: \xstar_i(\Expa) = \cFi_i(\xstar(\Expa);\Expa) =0 \right\}
\end{align}
Note that if $\mer$ is a C-function, $\mer^2$ is also a C-function since $\mer^2 = 0 \iff \mer = 0$. We observe from propositions \ref{prop:uncminNess} and \ref{prop:uncminSuff} that  minimizing $\f(\x;\randa)$ over $\x$  is equivalent to solving the NCP in \eqref{eq:NCP}.

Now we assume conditions on the smoothness of $\cF$ so that the solution to a perturbed problem is sufficiently close to the original solution.
\begin{assumption}\label{as:differen}
	$\cF(\x;\randa)$ is twice continuously differentiable in $\x$ and $\randa$ over an open set containing $\mathbb{K}$. 
\end{assumption}

Given that the rest of the analysis is on the function $\f(\x;\randa)$ defined in \eqref{eq:f}, we prove that a sufficiently smooth $\cF$ and a suitable $\mer$ ensure a sufficiently smooth $\f$.
\begin{proposition} \label{prop:diff}
	With assumption \ref{as:differen} holding, we state $\f(\x;\Expa)$ defined as in \eqref{eq:f} is a twice continuously differentiable at $\x$ satisfying $\f(\x;\Expa)=0$ for any C function $\mer$ provided
	\begin{enumerate}
		\item $\mer$ is twice differentiable at $\left\{ (a,b)\in \mathbb{R}^2: \mer(a,b)=0 \right\} \setminus  \{(0,0)\}$ with a finite derivative and finite second derivative.
		\item $\mer$ vanishes sufficiently fast near the origin. \ie \begin{align}
		\lim_{(a,b)\rightarrow(0,0)} \mer^2(a,b) \frac{\partial^2 \mer(a,b) }{\partial a\partial b} \quad &= \quad 0
		\end{align}
	\end{enumerate}
\end{proposition}
\begin{proof}
	Given that $\f$ is a sum of squares, it is sufficient to prove each term individually is twice continuously differentiable to prove the theorem. Also since we are only interested where $\f$ vanishes, it is sufficient to prove the above property for each term where it vanishes.\\
Consider terms from $i\not \in \mathcal{I}$. Since $\cFi_i$ is twice continuously differentiable, $\cFi_i^2$ is twice continuously differentiable too.\\
Consider the case $i\in \mathcal{Z}$. This means $i\in \mathcal{I} $ and $ \xstar_i(\Expa) = \cFi_i \left( \xstar_i(\Expa),\Expa \right)  = 0 $. These contribute a $\mer^4$ term to $\f$. With the notation $\mer_i \equiv \mer(\x_i,\cFi_i(\x;\randa))$, and $\delta_{ij} = 1 \iff i=j$ and $0$ otherwise we clearly have,
\begin{align}
\frac{\partial \mer^4_i}{\partial \x_j}  \quad &= \quad 4\mer^3_i \left( \frac{\partial \mer_i}{\partial a}\delta_{ij} + \frac{\partial \mer_i}{\partial b} \frac{\partial \cFi_i}{\partial \x_j}    \right) \quad = \quad 0\\
\frac{\partial^2 \mer^4_i}{\partial \x_j\x_k}  \quad &= \quad  12\mer^2_i \left(\frac{\partial \mer_i}{\partial a}\delta_{ij} + \frac{\partial \mer_i}{\partial b}\frac{\partial \cFi_i}{\partial \x_j}  \right)\left(\frac{\partial \mer_i}{\partial a}\delta_{ik} + \frac{\partial \mer_i}{\partial b}\frac{\partial \cFi_i}{\partial \x_k}  \right) \nonumber\\%
\quad & \qquad \quad + 4\mer^3_i \left( \frac{\partial^2 \mer_i}{\partial a^2}\delta_{ij}\delta_{ik} + \frac{\partial^2 \mer_i}{\partial a\partial b}\frac{\partial \cFi_i}{\partial \x_k} + \mbox{ other terms }    \right) \\
\quad &= \quad 0
\end{align}
For the third case, $i\in \mathcal{I}\setminus \mathcal{Z}$, we have
\begin{align}
\frac{\partial \mer_i^2}{\partial \x_j}\quad &= \quad 2\mer_i \left(\frac{\partial\mer_i }{\partial a}\delta_{ij} +\frac{\partial \mer_i}{\partial b}\frac{\partial \cFi_i}{\partial \x_j}   \right) \quad = \quad 0\\
\frac{\partial^2\mer_i^2}{\partial \x_j\x_k}\quad &= \quad  \left(\frac{\partial\mer_i }{\partial a}\delta_{ij} +\frac{\partial \mer_i}{\partial b}\frac{\partial \cFi_i}{\partial \x_j}   \right) \nonumber\\
\quad & \qquad \quad + 2\mer_i \left( \frac{\partial^2 \mer_i}{\partial a^2}\delta_{ij}\delta_{ik} + \frac{\partial^2 \mer_i}{\partial a\partial b}\frac{\partial \cFi_i}{\partial \x_k} + \mbox{ other terms }    \right) \\
\quad &= \quad \frac{\partial\mer_i }{\partial a}\delta_{ij} +\frac{\partial \mer_i}{\partial b}\frac{\partial \cFi_i}{\partial \x_j}
\end{align}
Continuity of $\f$ at the points of interest follow the continuity of the individual terms at the points.
\end{proof}

The following corollaries show the existence of  C-functions $\mer$ which satisfy the hypothesis of proposition \ref{prop:diff}. 
\begin{corollary}\label{corr:mindiff}
	With assumption \ref{as:differen} holding, for the choice of C-function $\mer = \mer_{min}$ defined in \eqref{eq:Min}, the function $\f$ is twice continuously differentiable at its zeros.
\end{corollary}
\begin{corollary}\label{corr:FBdiff}
	With assumption \ref{as:differen} holding, for the choice of C-function $\mer = \mer_{FB}$ defined in \eqref{eq:FisBurm}, the function $\f$ is twice continuously differentiable.
\end{corollary}

We now define an isolated solution to a problem and assume that the problem of interest has this property. This is required to ensure that our approximation is well defined.
\begin{definition}\label{def:isol} \citep{nocedal2006numerical}
	A minimum $\xstar$ of a problem is said to be an \emph{isolated minimum}, if there is a neighborhood $\mathcal{B}(\xstar;\epsilon)$ of $\xstar$, where $\xstar$ is the only minimum of the problem.
\end{definition}

\begin{figure}[t]
    \centering
    \begin{subfigure}[b]{0.45\textwidth}
        \includegraphics[width=\textwidth]{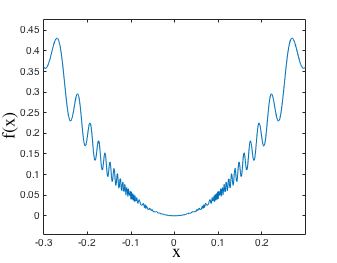}
        \caption{An example of a function where the global minimum $x=0$ is a non-isolated solution}\label{fig:intuitiona}
    \end{subfigure}
    \begin{subfigure}[b]{0.45\textwidth}
        \includegraphics[width=\textwidth]{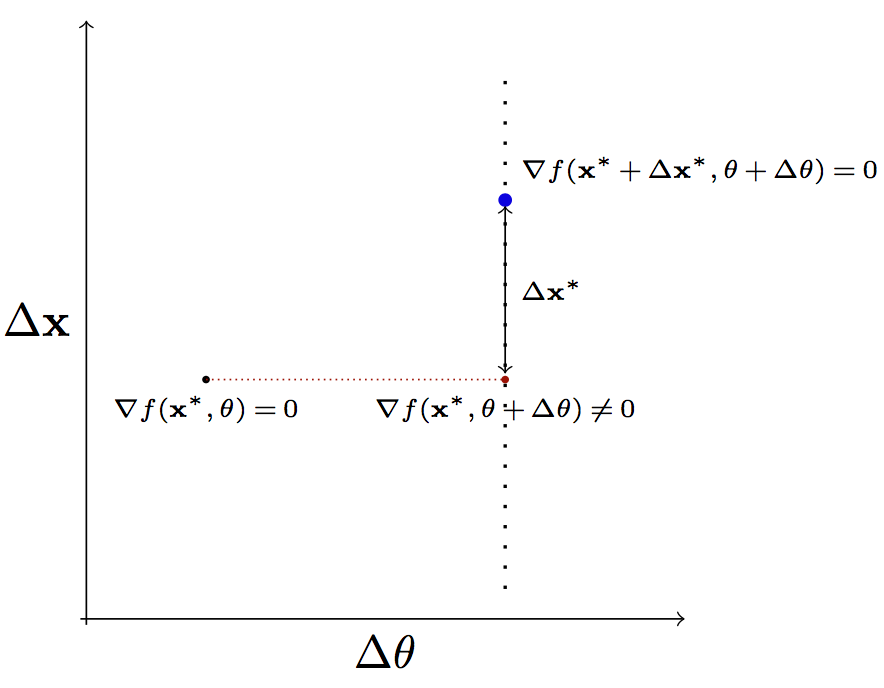}
        \caption{The intuition behind our approximation for finding where $\nabla \f(\x,\randa)=0$ under a small perturbation}\label{fig:intuitionb}
    \end{subfigure}
    \caption{Intuitions}
    \label{fig:intuition}
\end{figure}

A counter-example for an isolated minimum is shown on  {Fig. \ref{fig:intuitiona}}. It is a plot of the function
\begin{align}
f(x)\quad &= \quad 5x^2 + x^2\sin \left( \frac{1}{x^2}  \right)
\end{align}
and the global minimum at $x=0$ is not an isolated minimum as we can confirm that any open interval around $x=0$ has other minimum contained in it. Unlike this case, in this paper, we assume that if we obtain a global-minimum of $\f$, then it is an isolated minimum. The existence of a neighborhood $\mathcal{B}(\xstar;\epsilon)$ as required in definition \ref{def:isol} protects the approximation method from returning to a local minimum in the neighborhood for sufficiently small perturbations in problem parameters.
\begin{assumption}\label{as:isolated}
	For some fixed value of $\randa = \Expa$, there exists a known solution $\xstar(\Expa)$ such that it is an isolated global minimum of $\f(\x;\randa)$.
\end{assumption}

\subsection{Approximation Algorithm and error bounding} 

{This subsection achieves two primary results as follows. We now propose algorithm \ref{alg:Algoaprx} to approximate the covariance of the output given a covariance matrix for the input parameters. Following that Theorem \ref{thm:maintheorem} gives a mathematical proof that the algorithm indeed approximates the covariance matrix. The second key result is Theorem \ref{thm:bound} which bounds the error in the approximation. }

\SSDel{Given a deterministic shift in a parameter's value, the sensitivity of the solution has been studied in [20] for nonlinear optimization problems and in [7,8] in calculus of variations. This analysis is useful for analytic approaches, providing closed-form expressions for solution sensitivity. We aim to build on this research by proposing an approximation approach for large-scale stochastic complementarity problems, which do not necessarily overlap with the class of problems studied with calculus of variations. Our approach is amenable to extensions towards approximations for large-scale models, as we study the sensitivity of the variance of the solution, to a perturbation in the variance of input random parameters.} \sriAdd{Given a deterministic shift in a parameter’s value, the sensitivity of the solution has been studied in \cite{fiacco2009sensapprox,castillo2006} for mathematical programming problems (e.g., nonlinear or large-scale programs). Authors in \cite{castillo2008} used a perturbation technique to study the sensitivities in calculus of variations. We aim to build on the research in \cite{fiacco2009sensapprox,castillo2006} by proposing an approximation approach for large-scale stochastic complementarity problems. In particular, the sensitivity analysis in Chapter 3 in \cite{fiacco2009sensapprox} motivates the method developed in this paper, after converting the complementarity problem into an unconstrained optimization problem with sufficient smoothness properties. Chapters 4 and 5 in \cite{fiacco2009sensapprox} provide good context for computational approaches related to the ideas of Chapter 3. Our approach is amenable to extensions towards approximations for large-scale models, as we study the sensitivity of the variance of the solution, to a perturbation in the variance of input random parameters. We invite readers to refer to these works on further context in standard optimization problems. Chapter 2 of \cite{fiacco2009sensapprox} provides interesting examples on how small perturbations can lead to large changes in solutions. This is good context in situations where Assumptions 2 (and Corollaries 8 and 9) and 3 in our paper do not hold. The research in \cite{castillo2006} provides good context for standard optimization problems, including situations where we relax assumptions on smoothness and active constraints.}

{The intuition behind the approximation is shown on Fig. \ref{fig:intuitionb} and can be summarized as follows. Having posed the NCP as an unconstrained minimization of a function $\f$, we now approximate the change in the solution due to a perturbation of the parameters. Keeping the smoothness properties of $\f$ in mind, we say that the gradient of $\f$ vanishes at the solution before any perturbation. Following the random perturbation of parameters, we approximate the new value of the gradient of $\f$ at the old solution. Then we compute the step to be taken so that the gradient at the new point vanishes.  We formalize this idea in Theorem \ref{thm:maintheorem} and use that to build Algorithm \ref{alg:Algoaprx} with some features for increased efficiency.} The analysis builds on Dini's implicit function theorem \cite{krantz2012implicit} for deterministic perturbations and extends the results to predict covariance under random perturbations.

\begin{algorithm}[t]
	\caption{Approximating Covariance}
	  Solve the complementarity problem in \eqref{eq:NCP} for the mean value of $\randa= \Expa$, or solve the stochastic complementarity problem in $\eqref{eq:SCP}$ and calibrate the value of the parameters $\randa = \Expa$ for this solution. Call this solution as $\xstar$. Choose a tolerance level $\tau$. 
	\begin{algorithmic}[1]
	\State Evaluate ${\cF^*}\gets{\cF(\xstar;\Expa)}$, $G_{ij}\gets \frac{\partial \cFi_i(\xstar;\Expa)}{\partial \x_j} $, $L_{i,j}\gets \frac{\partial \cFi_i(\xstar;\Expa)}{\partial \randa_j} $.
	\State Choose a C-function $\mer$ such that the conditions in proposition \ref{prop:diff} are satisfied.
	\State Define the function $\mer^a(a,b) = \frac{\partial \mer(a,b)}{\partial a} $, $\mer^b(a,b) = \frac{\partial \mer(a,b)}{\partial b} $.
	\State Find the set of indices $\mathcal{Z} = \left\{ z\in \mathcal{I}: |\xstar_z| = |\cF^*_z| \leq \tau \right\}$.
	\State Define \begin{align}
	\nabPhi_{ij} \quad &\gets \quad  \left \{
	    \begin{array}{lr}
	        G_{ij}&\mbox{ if }i\not\in \mathcal{I}\\
	        0&\mbox{ if }i\in \mathcal{Z}\\
	        \mer^a(\xstar_i,\cFi^*_i)\delta_{ij} + \mer^b(\xstar_i,\cFi^*_i)G_{ij}
	        &\mbox{ otherwise}
	    \end{array}
	\right .
	\end{align}
	where $\delta_{ij} = 1$ if $i=j$ and 0 otherwise.
	\State Define \begin{align}
		\mathcal{N}_{ij}\quad &\gets \quad \left \{
		    \begin{array}{ll}
		        L_{ij}&\mbox{ if }i\not\in \mathcal{I}\\
		        0&\mbox{ if }i \in \mathcal{Z}\\
		        \mer^b(\xstar_i,\cFi_i^*)L_{ij}&\mbox{ otherwise}
		    \end{array}
		\right .
	\end{align}
	\State Solve the linear systems of equations for $\T$.
	\begin{align}
	\nabPhi\T \quad &= \quad \mathcal{N}
	\end{align}
	If $\nabPhi$ is non singular, we have a unique solution. If not, a least square solution or a solution obtained by calculating the Moore Penrose Pseudo inverse \citep{horn2012matrix} can be used.
	\State Given $\mathcal{C}$, a covariance matrix of the input random parameters, $\randa(\randomxi)$, \Return $\mathcal{C}^*  \gets \T \mathcal{C}\T^T$.
	\end{algorithmic}
	\label{alg:Algoaprx}
\end{algorithm}

\begin{theorem}\label{thm:maintheorem}
Algorithm \ref{alg:Algoaprx} generates Taylor's first-order approximation for the change in solution for a perturbation in parameters and computes the covariance of the solution for a complementarity problem with uncertain parameters with small variances.
\end{theorem}
\begin{proof}
	Consider the function $\f(\x;\randa)$. From Theorem \ref{prop:uncminNess}, $\xstar\equiv\xstar(\Expa)$ minimizes this function for $\randa = \Expa$. From proposition \ref{prop:diff}, we have $\f(\x;\randa)$ is twice continuously differentiable at all its zeros. Thus we have,
\begin{align}
\gradxf(\xstar;\Expa) \quad &= \quad 0
\end{align}
Now suppose the parameters $\Expa$ are perturbed by $\perturb$, then the above gradient can be written using the mean value theorem and then approximated up to the first order as follows.
\begin{align}
\gradxf(\xstar(\Expa),\Expa+\perturb)\quad &= \quad \gradxf(\xstar(\Expa),\Expa) \nonumber\\
\quad & \quad \qquad + \nabla_{\Expa}\gradxf(\xstar{(\Expa)},\widetilde{\randa})\perturb\\
\gradxf(\xstar(\Expa),\Expa+\perturb) - \gradxf(\xstar(\Expa),\Expa)\quad &\approx \quad \Jacob\perturb\label{eq:aprx1}\\
\mbox{where,}\quad& \nonumber\\
\widetilde{\randa} \quad &\in \quad [\randa,\randa+\perturb]\\
\Jacob_{ij}\quad &= \quad [\nabla_{\Expa}\gradxf(\xstar(\Expa),\Expa)]_{ij}\\
\quad &= \quad \frac{\partial [\gradxf(\xstar;\Expa)]_i}{\partial \Expa_j}\label{eq:Jacob}
\end{align}
Since $\Jacob\perturb$ is not guaranteed to be 0, we might have to alter $\x$ to bring the gradient back to zero. \ie we need $\Delta\x$ such that $\gradxf(\xstar(\Expa)+\Delta\x,\Expa+\perturb)= \mathbf{0}$. But by the mean value theorem,
\begin{align}
\gradxf(\xstar(\Expa)+\Delta\x,\Expa+\perturb)\quad &= \quad \gradxf(\xstar(\Expa),\Expa+\perturb) \nonumber\\
\quad & \quad \qquad + \hessxf(\widetilde{\x},\Expa+\perturb)\Delta\x\\
\mathbf{0}\quad &\approx \quad  \Jacob\perturb + \hessxf(\widetilde{\x},\Expa)\Delta\x\\
 \quad &\approx \quad \Jacob\perturb + \hessxf(\xstar(\Expa),\Expa)\Delta\x\label{eq:aprx2}\\
\Hstar\Delta\x \quad &\approx \quad -\Jacob\perturb\label{eq:linaprx}\\
\mbox{where,}\quad& \nonumber\\
\widetilde{\x}\quad &\in \quad [\xstar(\Expa),\xstar(\Expa)+\Delta\x]\\
[\Hstar]_{ij} \quad &= \quad [\hessxf(\xstar(\Expa),\Expa)]_{ij}\\
\quad &= \quad \frac{\partial [\gradxf(\xstar;\Expa)]_i}{\partial \x_j}\label{eq:Hstar}
\end{align}
Now from \cite{nocedal2006numerical}, the gradient of the least squares function $\f$ can be written as
\begin{align}
\gradxf(\xstar,\Expa)\quad &= \quad \nabPhi^T\Mer(\xstar,\Expa)\label{eq:Jdef}\\
[\nabPhi]_{ij}\quad &= \quad \frac{\partial \Mer_i(\xstar,\Expa)}{\partial \x_j}\\
\quad &= \quad \left\{ \begin{array}{ll}
	\frac{\partial \cF_i(\xstar,\Expa)}{\partial \x_j}\quad&\quad\mbox{if}\quad i\not\in \mathcal{I}\\
	\frac{\partial \mer^2(\x_i,\cF_i(\xstar,\Expa))}{\partial \x_j}\quad&\quad\mbox{if}\quad i\in \mathcal{Z}\\
	\frac{\partial \mer(\x_i,\cF_i(\xstar,\Expa))}{\partial \x_j}\quad&\quad\mbox{otherwise}\\
\end{array} \right .\\
\quad &= \quad \left\{ \begin{array}{ll}
	\frac{\partial \cF_i(\xstar,\Expa)}{\partial \x_j}\quad&\quad\mbox{if}\quad i\not\in \mathcal{I}\\
	0\quad&\quad\mbox{if}\quad i\in \mathcal{Z} \\
	\frac{\partial \mer_i}{\partial \x_j}\quad&\quad\mbox{otherwise}\\
\end{array} \right .
\end{align}
which is the form of $\nabPhi$ defined in algorithm \ref{alg:Algoaprx}. Also
\begin{align}
\Hstar \quad = \quad \hessxf(\xstar;\Expa) \quad &= \quad \nabPhi^T\nabPhi + \sum_{i=1}^n \Mer_i(\xstar;\Expa)\nabla_{\x}^2\Mer_i(\xstar;\Expa)\\
\quad &= \quad \nabPhi^T\nabPhi \label{eq:Htemp1}
\end{align}
where the second term vanishes since we have from Theorem \ref{prop:uncminNess} that each term of $\Mer$ individually vanishes at the solution. Now
\begin{align}
\Jacob \quad &= \quad \nabla_{\x\randa}\f(\xstar;\Expa)\\
\Jacob_{ij}\quad &= \quad \frac{\partial [\gradxf(\xstar;\Expa)]_i}{\partial \randa_j} \\
\quad &= \quad \frac{\partial }{\partial \randa_j}\left( \sum_{k=1}^n [\nabla_{\x}\Mer(\xstar;\Expa)]_{ki}\Mer_k(\xstar;\Expa)  \right) \\
\quad &= \quad \sum_{k=1}^n \left( \frac{\partial [\nabla_{\x}\Mer(\xstar;\Expa)]_{ki}}{\partial \randa_j} \Mer_k(\xstar;\Expa) + [\nabla_{\x}\Mer(\xstar;\Expa)]_{ki} \frac{\partial \Mer_k(\xstar;\Expa)}{\partial \randa_j}  \right) \\
\quad &= \quad \sum_{k=1}^n \nabPhi_{ki}\mathcal{N}_{kj} \quad = \quad \nabPhi^T \mathcal{N}\label{eq:Jtemp1}
\end{align}
where the first term vanished because $\Mer_i$ are individually zeros, and we define
\begin{align}
\mathcal{N}_{ij} \quad &= \quad \frac{\partial \Mer_i(\xstar;\Expa)}{\partial \randa_j}\\
\quad &= \quad \left \{
    \begin{array}{lr}
        \frac{\partial \cFi_i}{\partial \x_j} &\mbox{ if }i \not\in \mathcal{I}\\
        2\mer(\xstar_i;\cFi_i^*)\mer^b(\xstar_i;\cFi_i^*)\frac{\partial \cFi_i}{\partial \randa_j}&\mbox{ if }i\in \mathcal{Z} \\
        \mer^b(\xstar_i;\cFi_i^*)\frac{\partial \cFi_i(\xstar;\Expa)}{\partial \randa_j}&\mbox{ otherwise }\\
    \end{array}
\right .\\
\quad &= \quad \left \{
    \begin{array}{ll}
        \frac{\partial \cFi_i}{\partial \randa_j} &\mbox{ if }i \not\in \mathcal{I}\\
        0&\mbox{ if }i\in \mathcal{Z} \\
        \mer^b(\xstar_i;\cFi_i^*)\frac{\partial \cFi_i(\xstar;\Expa)}{\partial \randa_j}&\mbox{ otherwise }\\
    \end{array}
\right .
\end{align}
which is the form of $\mathcal{N}$ defined in algorithm \ref{alg:Algoaprx}. By assumption \ref{as:isolated}, we have a unique minimum in the neighborhood of $\xstar$ where the gradient vanishes. So we have from \eqref{eq:linaprx}, \eqref{eq:Htemp1} and \eqref{eq:Jtemp1}
\begin{align}
\Hstar\Delta\x \quad &= \quad -\Jacob\perturb \\
\nabPhi^T\nabPhi\Delta\x \quad &= \quad -\nabPhi^T\mathcal{N}\perturb
\end{align}
$\Delta \x$ solves the above equation, if it solves
\begin{align}
\nabPhi \Delta\x \quad &= \quad -\mathcal{N}\perturb
\end{align}
By defining $\T$ as the solution to the linear system of equations
\begin{align}
\nabPhi\T \quad &= \quad \mathcal{N}\label{eq:TheResult}\\
\Delta\x \quad &= \quad -\T\perturb
\end{align}
and we have the above first-order approximation. From \cite{seber1989}, we know that if some vector $\x$ has covariance $\C$, then for a matrix $A$, the vector $A\x$ will have covariance $A\C A^T$. So we have.
\begin{align}
\cov(\Delta\x) \quad &\approx \quad \T\cov(\perturb)\T^T
\end{align}
Thus we approximate the covariance of $\Delta\x$ in algorithm \ref{alg:Algoaprx}.
\end{proof}

{The matrix $\nabPhi$ could potentially not have full rank, for example, when $\mathcal{Z}$ is non-empty, i.e., when we have weak complementarity terms. But we note that $\nabPhi^T\nabPhi$ is the Hessian of $\f$ and the null-space of the Hessian corresponds to the directions where the gradient doesn't change for small perturbations. So in a first-order sence, small perturbations in those directions do not move the solution, keeping the method robust even under weak complementarity.} {This is further confirmed by computational experiments detailed in Section 5.2 and Appendix D where we have cases with weak complementarity terms but no significant error.}

{Further, when $\nabPhi$ is singular, we use Moore-Penrose pseudoinverse, $\nabPhi^{\dagger}$ to solve the system of equations \eqref{eq:TheResult}. Among possibly infinite solutions that minimize the error $\left \Vert \nabPhi\T-\mathcal{N}\right \Vert_2$, $\nabPhi^{\dagger}\mathcal{N}$ gives the solution that minimizes $\left \Vert \T\right \Vert_2$ \citep{ben2003generalized}. This would lead us to identifying the smallest step $\Delta\x$ that could be taken to reach the perturbed solution, up to first-order approximation and hence give the most conservative estimate of the uncertainty.  Uniqueness and existence of $\nabPhi^{\dagger}$ is guaranteed and it can be computed efficiently.  }For computational purposes the matrix $\T$ in the above equation has to be calculated only once, irrespective of the number of scenarios for which we would like to run for the covariance of $\randa$. Thus if $\x \in \mathbb{R}^n$, $\randa \in \mathbb{R}^m$ and we want to test the output covariance for $k$ different input covariance cases, the complexity is equal to that of solving a system of $n$ linear equations $m$ times as in \eqref{eq:TheResult}, and hence is $O(mn^2)$. \ie the complexity is quadratic in the number of output variables, linear in the number of input parameters and constant in the number of covariance scenarios we would like to run.

In Theorem \ref{thm:bound} below, we prove that the error in the approximation of theorem \ref{thm:maintheorem} can be bounded using the condition number of the Hessian. We need the following assumption that the condition number of the Hessian of $\f$ is bounded and the Hessian is Lipschitz continuous.

\begin{assumption}\label{as:condnNos}
	At the known solution of the complementarity problem of interest ($\randa = \Expa$),
	\begin{enumerate}
		\item The condition number of the Hessian of $\f$ defined is finite and equals to $\kappa_H$
		\item The Hessian of $\f$ is Lipschitz continuous with a Lipschitz constant $\Lipsh$.
	\end{enumerate}
\end{assumption}

\begin{theorem}\label{thm:bound}
	With assumption \ref{as:condnNos} holding, the error in the linear approximation \ref{eq:linaprx} for a perturbation of $\epsilon$ is $o(\epsilon)$.
\end{theorem}
\begin{proof}
	Since $\nabla^2f$ is Lipschitz continuous on both $\x$ and $\randa$, we can write for $\widetilde{\x}$ near $\xstar$,
	\begin{align}
	\left\Vert \hessxf(\xstar,\Expa) - \hessxf(\widetilde{\x},\Expa) \right\Vert \quad &\leq \quad \Lipsh \left\Vert \xstar - \widetilde{\x} \right\Vert \\
	\quad &\leq \quad \Lipsh\left\Vert \Delta\x \right\Vert\\
	\widetilde{\Hstar} \quad &= \quad \hessxf(\widetilde{\x},\Expa) \\
	\quad &= \quad \Hstar + \varepsilon_H
	\end{align}
	where $\left\Vert \varepsilon_H \right\Vert \leq \Lipsh \left\Vert \Delta\x \right\Vert$. Applying the Lipschitz continuity on $\randa$,
	\begin{align}
		\left\Vert \nabla_{\randa}\gradxf(\xstar,\widetilde{\randa}) - \nabla_{\randa}\gradxf(\xstar,\Expa) \right\Vert \quad &\leq \quad \Lipsh\left\Vert \widetilde{\randa} - \Expa\right\Vert\\
		\quad &\leq \quad  \Lipsh\left\Vert \perturb\right\Vert\\		
		\widetilde{\Jacob} \quad &= \quad \nabla_{\randa}\gradxf(\xstar,\widetilde{\randa})\\
		\quad &= \quad \Jacob + \varepsilon_J
	\end{align}
	where $\left\Vert \varepsilon_J \right\Vert \leq \Lipsh \left\Vert \Delta\theta \right\Vert$. Thus the equation
	\begin{align}
	\widetilde{\Hstar}\Delta\x \quad &= \quad \widetilde{\Jacob}\perturb
	\end{align}
	is exact, even if we cannot compute $\widetilde{\Hstar}$ and $\widetilde{\Jacob}$ exactly. Now the error in inverting $\widetilde{\Hstar}$ is bounded by the condition number \citep[Ch. ~5]{horn2012matrix}.
	\begin{align}
	\frac{\left\Vert \Hstar^{-1} - \widetilde{\Hstar}^{-1} \right\Vert}{\left\Vert \widetilde{\Hstar}^{-1} \right\Vert} \quad &\leq \quad \frac{\kappa_H\frac{\left\Vert \varepsilon_H \right\Vert}{\left\Vert \widetilde{\Hstar} \right\Vert}}{1-\kappa_H\frac{\left\Vert \varepsilon_H \right\Vert}{\left\Vert \widetilde{\Hstar} \right\Vert}}
	\end{align}
	Assuming $\kappa_H\left\Vert \varepsilon_H \right\Vert \ll \left\Vert \Hstar \right\Vert$, the above equation becomes
	\begin{align}
	\frac{\left\Vert \Hstar^{-1} - \widetilde{\Hstar}^{-1} \right\Vert}{\left\Vert \widetilde{\Hstar}^{-1} \right\Vert} \quad &\leq \quad \kappa_H\frac{\left\Vert \varepsilon_H \right\Vert}{\left\Vert \widetilde{\Hstar} \right\Vert}\\
	\Rightarrow \left\Vert \widetilde{\Hstar}^{-1}-{\Hstar}^{-1} \right\Vert \quad &\leq \quad   \kappa_H  \frac{\left\Vert \widetilde{\Hstar}^{-1} \right\Vert}{\left\Vert \widetilde{\Hstar} \right\Vert}\left\Vert \varepsilon_H \right\Vert \label{eq:condTemp1}\\	
	\Rightarrow  \left\Vert \widetilde{\Hstar}^{-1}\widetilde{\Jacob} - \Hstar^{-1} {\Jacob} - \Hstar^{-1}\varepsilon_J  \right\Vert  \quad &\leq \quad    	\kappa_H  \frac{\left\Vert \widetilde{\Hstar}^{-1} \right\Vert}{\left\Vert \widetilde{\Hstar} \right\Vert}\varepsilon_H \left\Vert \Jacob \right\Vert + \kappa_H  \frac{\left\Vert \widetilde{\Hstar}^{-1} \right\Vert}{\left\Vert \widetilde{\Hstar} \right\Vert}\left\Vert \varepsilon_H \right\Vert \left\Vert \varepsilon_J \right\Vert\\
	\Rightarrow \left\Vert \widetilde{\Hstar}^{-1}\widetilde{\Jacob} - \Hstar^{-1}\Jacob \right\Vert \quad &\leq \quad k_1 \left\Vert \Delta\x\right\Vert + k_2 \left\Vert \perturb \right\Vert\label{eq:bound}\\
	\text{with}\qquad & \nonumber\\
	k_1 \quad &= \quad \kappa_H \Lipsh \frac{\left\Vert \widetilde{\Hstar}^{-1} \right\Vert}{\left\Vert \widetilde{\Hstar} \right\Vert}\left\Vert \Jacob \right\Vert \\
	k_2 \quad &= \quad \Lipsh\left\Vert \Hstar^{-1} \right\Vert
	\end{align}
Thus we have from \eqref{eq:bound}, that the error in the approximation done in algorithm \ref{alg:Algoaprx} is bounded.
\end{proof}

\section{Stochastic Sensitivity Analyses} 
\label{sec:total_sensitivity}
In this section, we quantify the sensitivity of the variance of the solution to variance in each of the input parameters.  To achieve this, we define total linear sensitivity and how it can be approximated using the matrix $\T$ derived {in \eqref{eq:TheResult} }. We then proceed to prove that these quantities also bound the maximum increase in uncertainties of the output.
\begin{definition}\label{def:LinearSensitivity}
	Given a function $\f: \mathbb{R}^m\mapsto \mathbb{R}^n$, the \emph{total linear sensitivity}, $\sens_d \in \mathbb{R}_{+}$ of a dimension $d\leq m; \, d\in \mathbb{N}$ at a point $\x \in \mathbb{R}^m$ is defined for $\delta>0$, sufficiently small,
	\begin{align}
	\sens_d \quad &= \quad \inf \left\{ \alpha: \left\vert\vphantom{\int}\left\Vert \f(\x+\delta e_d) \right\Vert_2 - \left\Vert \f(\x) \right\Vert_2\right\vert\leq \delta\alpha + o \left( \delta^2 \right)  \right\}
	\end{align}
	where $e_d$ is the $d$-th standard basis vector.
\end{definition}
This is a bound on the distance by which the function value can move for a small perturbation in the input. Now we look at the solution to the parametrized complementarity problem in \eqref{eq:NCP} as a function from the space of parameter tuples to the space of solution tuples and bound the change in solution for a small perturbation. The next proposition shows how the total linear sensitivity can be calculated from the linear approximation matrix $\T$ derived earlier.
\begin{proposition}\label{prop:sens1}
Suppose we know, $G \in \mathbb{R}^{n\times m}$ such that $G_{ij} = \frac{\partial \f_i(\x)}{\partial \x_j} $, then $\sens_d = \sqrt{\left( \sum_{i=1}^n G_{id}^2 \right) }$
\end{proposition}
\begin{proof}
	See Appendix A
\end{proof}

The above proposition proves that the $\T$ matrix obtained in \eqref{eq:TheResult} is sufficient to approximate the total linear sensitivity. The following result suggests how the total linear sensitivity can approximate the total variance in the output variables.
\begin{theorem}
	Given a function $\f:\mathbb{R}^m \mapsto \mathbb{R}^n$ and $\sens_d$, the increase in the total uncertainty in the output, \ie the sum of variances of the output variables, for a small increase of the variance of an input parameter, $\sigma_d^2$ of $\x_d$ is approximated by $\sens_d^2\sigma_d^2$.
\end{theorem}
\begin{proof}
	Let $E_d$ be the matrix of size $m\times m$ with zeros everywhere except the $d$-th diagonal element, where it is 1. Given $C = \cov(\x(\randomxi))$, for a small perturbation $\sigma^2$ in the variance of $\x_d$, the covariance of $\f(\x)$ changes as follows.
\begin{align}
C^* \quad &\approx \quad \gradxf C\gradxf^T\\
C^* + \Delta C^* \quad &\approx \quad \gradxf (C+\sigma^2 E_d)\gradxf^T\\
\quad &= \quad C^* + \sigma^2\gradxf E_d\gradxf^T\\
[\Delta C^*]_{ij} \quad &\approx \quad \sigma^2 [\gradxf]_{id}[\gradxf]_{jd}\\
\sum_{i=1}^n [\Delta C^*]_{ii} \quad &\approx \quad \sigma^2\sens_d^2
\end{align}
which is the total increase in variance. The off-diagonal terms do not affect the total uncertainty in the system because, the symmetric matrix $C$ can be diagonalized as $QDQ^T$, where $Q$ is a rotation matrix, and the trace is invariant under orthogonal transformations.
\end{proof}
With the above result, we can determine the contribution of each input parameter to the total uncertainty in the output. 

\section{Application to optimization} 
\label{sec:application_to_unconstrained_minimization}
To illustrate the application of this method explained in algorithm \ref{alg:Algoaprx}, we use it to derive an approximation for the covariance of the solution of certain canonical optimization problems. The goal of this section is to walk the reader through a simple application of the method to develop intuition of the analysis and results. 

To start with, we assume conditions on the differentiability and convexity of the objective function.
\begin{assumption}\label{as:objfun}
The objective function $\f(\x;\randa)$ is strictly convex in $\x$ and is twice continuously differentiable in $\x$ and $\randa$.
\end{assumption}

In the theorem below, we approximate the covariance of the decision variables of a convex optimization with uncertainties in the linear term and with only linear equality constraints.
\begin{theorem}
	With assumption \ref{as:objfun} holding, the covariance of the primal and dual variables at the optimum of the problem,
	\begin{align}
	\Min_{\x} \quad &\f(\x;\randa) = g(\x)+c(\randa)^T\x\\
	\text{subject to }&Ax = b(\randa)&(\y)
	\end{align}
	where $\randa = \randa(\randomxi)$ are random parameters with covariance $C$, is first-order approximated by $\T C\T^T$ where
	\begin{align}
	\T \quad &= \quad \left( \begin{array}{cc}
	\nabla_{\x}^2g(\xstar)&A^T\\
	A&0
	\end{array}	 \right) ^{-1} \left( \begin{array}{c}
	-\nabla_{\randa}c(\randa)\\
	\nabla_{\randa}b(\randa)
	\end{array} \right)
	\end{align}
\end{theorem}
\begin{proof}
	For the given optimization problem, because of assumption \ref{as:objfun} and linear independence constraint qualification (LICQ), the KKT conditions are necessary and sufficient for optimality. The KKT condition satisfied at a solution $(\xstar,\ystar)$ for the problem are given by
\begin{align}
\nabla_{\x}g(\xstar) + c(\randa) + A^T\ystar \quad &= \quad  0\\
A\xstar \quad &= \quad b(\randa)
\end{align}
for some vector $\y$ so that the equation is well defined. Suppose from there, $\randa$ is perturbed by $\perturb$, we have
\begin{align}
\nabla_{\x}g(\xstar) + c(\randa+\perturb) + A^T\ystar \quad &\approx \quad  \nabla_{\randa}c(\randa)\perturb \\
A\xstar -b(\randa+\perturb)	 \quad &\approx \quad -\nabla_{\randa}b(\randa) \perturb
\end{align}
Now we need to find $\Delta\x$ and $\Delta\y$ such that
\begin{align}
\nabla_{\x}g(\xstar+\Delta\x) + c(\randa+\perturb) + A^T(\ystar+\Delta\y) \quad &\approx \quad  0\\
A(\xstar+\Delta\x) -b(\randa+\perturb)	 \quad &\approx \quad 0\\
\nabla_{\x}^2g(\xstar)\Delta\x + A^T\Delta\y \quad &\approx \quad \nabla_{\randa}c(\randa)\perturb\\
A\Delta\x \quad &\approx \quad -\nabla_{\randa}b(\randa)\perturb
\end{align}
The above conditions can be compactly represented as
\begin{align}
\left( \begin{array}{cc}
	\nabla_{\x}^2g(\xstar)&A^T\\
	A&0
	\end{array}	 \right) \left( \begin{array}{l}
	\Delta\x\\ \Delta\y
	\end{array} \right) \quad &= \quad \left( \begin{array}{l}
	\nabla_{\randa}c(\randa)\\ -\nabla_{\randa}b(\randa)
	\end{array} \right) \perturb
\end{align}
If $A$ has full rank, then the above matrix is non-singular. So the change in the decision variables $\x$ and the duals $\y$ can be written as a linear transformation of the perturbation in the random parameters. And we now have
\begin{align}
\cov \left( \begin{array}{l}
	\Delta\x\\ \Delta\y
	\end{array}  \right) \quad &= \quad \T \cov(\randa)\T^T\\
	\T \quad &= \quad \left( \begin{array}{cc}
	\nabla_{\x}^2g(\xstar)&A^T\\
	A&0
	\end{array}	 \right) ^{-1} \left( \begin{array}{c}
	-\nabla_{\randa}c(\randa)\\
	\nabla_{\randa}b(\randa)
	\end{array} \right)
\end{align}
\end{proof}
In the corollary below, we show that the method suggested is accurate (\ie has zero error) for an unconstrained quadratic optimization problem with uncertainty in the linear term.
\begin{corollary}\label{thm:QP}
	For an optimization problem with uncertainty of objectives of the form,
	\begin{align}
		f(\x;\randa) \quad &= \quad \frac{1}{2}\x^TG\x + \randa(\randomxi)^T\x
	\end{align}
	where $G$ is positive definite, the approximation method has zero error. In other words, the obtained covariance matrix is exact.
\end{corollary}
\begin{proof}
	See Appendix A
\end{proof}

\section{Application to a general oligopoly market}
\label{Sec:Duopoly}
We now present an example of a complementarity problem in a natural gas oligopoly and show how the methods developed in this paper can be applied.

\subsection{Problem Formulation and results} 
\label{sub:problem_formulation_and_results}
Consider $k$ producers competitively producing natural gas in a Nash-Cournot game. Let the random unit costs of production be $\gamma_i(\randomxi), \quad i\in \left\{ 1,\ldots,k \right\}$. Also, let us assume that the consumer behavior is modeled by a linear demand curve $P(\tilde{Q})$ as follows.
\begin{align}
P \quad &= \quad a(\randomxi) + b(\randomxi)\tilde{Q}
\end{align}
where $P$ is the price the consumer is willing to pay, $\tilde{Q}$ is the total quantity of the natural gas produced and random variables $a(\randomxi) > 0,\, b(\randomxi)<0\, \forall \randomxi\in\Omega $. Suppose the producers are maximizing their profits, then the 
Nash equilibrium can be obtained by solving the following complementarity problem \citep{cottle2009linear,gabriel2012complementarity}.
\begin{align}
0\leq Q_i\perp\cFi_i \left(  \mathbf{Q}  \right) =  \gamma_i-a -b \left( \sum_{j=1}^kQ_k \right) -bQ_i\geq 0 \label{eq:Oligopoly}
\end{align}
In this formulation, $a,b,\gamma_i$ correspond to $\randa$ and $Q_i$ correspond to $\x$ in \eqref{eq:NCP} with $\mathcal{I}=\left\{ 1,2,\ldots,k \right\}$. In the current numerical example, let us consider a duopoly where $k=2$. Let
\begin{align}
\expec \left( \begin{array}{cccc}
\gamma_1&
\gamma_2&
a&
b
\end{array} \right)^T \quad &= \quad  \begin{pmatrix}
2&1&15&-1
\end{pmatrix} ^T
\end{align}
Solving the complementarity problem deterministically with the above parameter values, we get $Q_1$ and $Q_2$ to be 4 and 5 respectively. We use the C-function $\mer_{min}(x,y) = \min(x,y)$ for this example to get
\begin{alignat}{3}
\nabPhi \quad &= \quad \left( \begin{array}{cc}
2&1\\
1&2
\end{array} \right) &\qquad& \mathcal{N}\quad &= \quad \left(
\begin{array}{cccc}
1&0&-1&-13\\
0&1&-1&-14
\end{array}
 \right)
\end{alignat}
Now we have from \eqref{eq:TheResult}
\begin{alignat}{3}
\T \quad &= \quad \nabPhi^{-1}\mathcal{N}& \qquad &
\quad &= \quad \frac{1}{3}\left( \begin{array}{cccc}
2&-1&-1&-12\\
-1&2&-1&-15
\end{array} \right)
\end{alignat}
Having obtained $\T$, we attempt to get insight on how uncertainties in various input parameters propagate through the model causing uncertainty in the equilibrium quantities. If we assume that all these parameters, viz. $\gamma_1,\gamma_2,a,b$ have a 10\% coefficient of variation and are all uncorrelated, then the covariance matrix of the input is
\begin{align}
C_1 \quad &= \quad \left( \begin{array}{cccc}
0.04&0&0&0\\
0&0.01&0&0\\
0&0&2.25&0\\
0&0&0&0.01
\end{array} \right)
\end{align}
Then the covariance matrix of the solution would be
\begin{alignat}{3}
C_1^* \quad &= \quad \T C_1\T^T& \qquad&
\quad &= \quad \left( \begin{array}{cc}
0.4289&0.4389\\
0.4389&0.5089
\end{array} \right)
\end{alignat}
The standard deviation of the produced quantities are 0.65($=\sqrt{0.4289}$) and 0.71($=\sqrt{0.5089}$) respectively. The produced quantities also have about 95\% positive correlation as an increase in demand will cause both producers to produce more and a decrease in demand will cause both producers to produce less.

If we assume that we have perfect knowledge about the demand curve, and if the uncertainty is only in the production costs, then the new parameter covariance $C_2$ has the third and fourth diagonal term of $C_1$ as zero. In such a scenario, we would expect the decrease in the quantity of production of one player to cause an increase in the quantity of production of the other and vice versa, caused by re-adjustment of market share. We can see this effect by computing the covariance of the solution as $\T C_2\T$. The solution thus obtained shows that the produced quantities are negatively correlated with a correlation of $-85\%$. The uncertainties in the produced quantities are 3\% and 2\% respectively of the quantity produced by each producer. We also note that the variances are smaller now, as we no longer have uncertainties stemming from the demand side of the problem.

Now if we assume a more realistic scenario of the production costs being correlated (60\% correlation), then we note that the produced quantity are negatively correlated with  $-62\%$ correlation. The standard deviations in the produced quantities have also dropped to about 2.9\% and 1.2\%. Thus we not only obtain insight about the uncertainties in the output, but also the correlation between the output parameters. From an energy market policy maker's perspective this is crucial information as it helps identifying the regions where increase or decrease in production, consumption, price, pipeline flows and infrastructural expansions occur simultaneously and where they change asynchronously. 
Now we calculate the sensitivity of each of the input parameters to identify the parameter that causes maximum uncertainty in the output. The values for $\sens$ for each of the four parameters $\gamma_1,\,\gamma_2,\,a,\,b$ are calculated below.
\begin{align}
\sens \quad &= \quad \frac{1}{3}\begin{pmatrix}
	\sqrt{5} && \sqrt{5} && \sqrt{2} &&\sqrt{369}
\end{pmatrix}^T \nonumber\\
 \quad &= \quad \begin{pmatrix}
	0.745&&0.745&&0.471&&6.40
\end{pmatrix}^T
\end{align}
Thus we see that the solution is more sensitive to the slope of the demand curve than to say production cost. Strictly speaking, if we define the variance in equilibrium as the sum of the variance of all output variables, this says, a unit increase in variance of the slope of the demand curve will be magnified about 41 times $(6.4^2)$ variance in the equilibrium. {However, a unit increase in the variance of the production cost only increases the variance in equilibrium by 0.556 ($0.745^2$) units.} 
\subsection{Computational Complexity} 
\label{sub:computational_complexity}
We used a Monte-Carlo based method as a comparison against our approximation method to compute covariance in the decision variables. {To achieve this, we modeled the oligopoly complementarity problem mentioned in \eqref{eq:Oligopoly} varying the number of players, and hence the number of random parameters and the decision variables.} For the Monte-Carlo simulation based approach, a symmetrically balanced stratified design \citep{shields2015refined} is used with each dimension divided into two strata. With increasing number of random parameters and equilibrium variables, Monte-Carlo methods become increasingly inefficient as the number of simulations required grows exponentially. A comparison of the time taken in an \emph{8GB RAM 1600 MHz DDR3 2.5GHz Intel Core i5} processor to solve the above oligopoly problem with varying number of players is shown in Fig. \ref{fig:complexity}. Despite developments in algorithms to solve complementarity problems, the said exponential growth in the number of sample points required in a Monte-Carlo based approach deters the computational speed. A problem with as few as 25 uncertain variables takes about 2 hours to solve and one with 30 uncertain variables takes about seven days to solve using Monte-Carlo based approaches {while it takes few seconds to minutes in the first-order approximation method. Fig \ref{fig:MCSrace} compares the error between 5 rounds of Monte-Carlo simulation and the first-order approximation method. More details on these computational experiments are provided in Appendix D}.

\begin{figure}[t]
	\centering
    \begin{subfigure}[b]{0.45\textwidth}
 	\includegraphics[width=\textwidth]{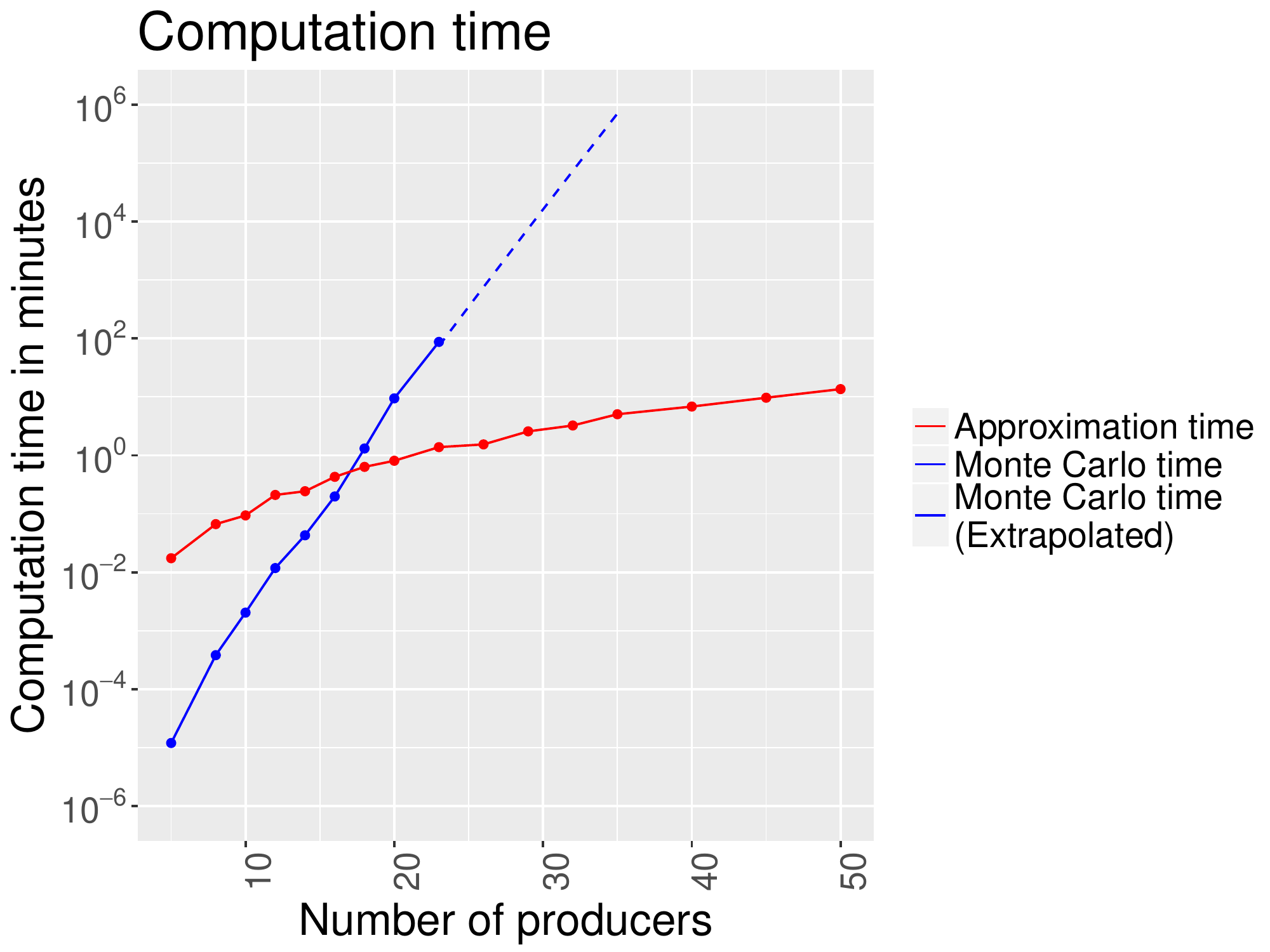}\caption{Run time comparison}\label{fig:complexity}
    \end{subfigure}
    \begin{subfigure}[b]{0.45\textwidth}
        \includegraphics[width=\textwidth]{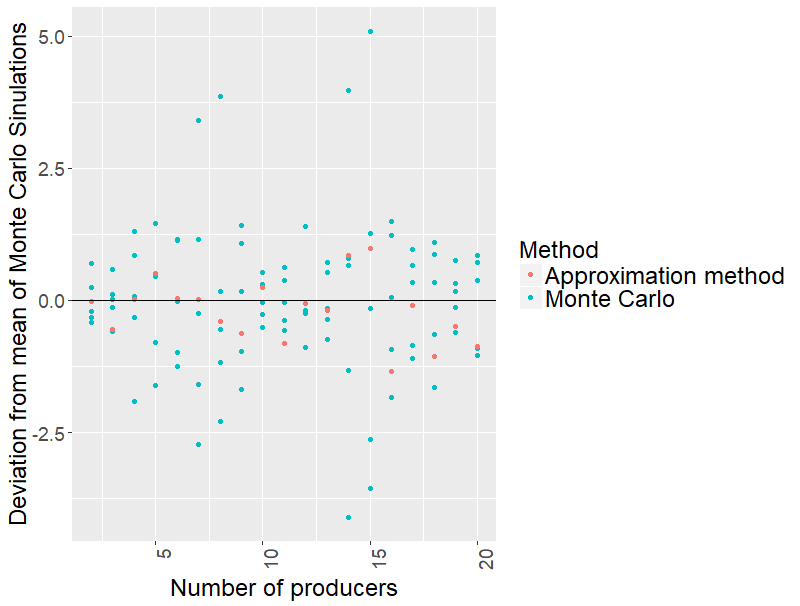}\caption{Error comparison}\label{fig:MCSrace}
    \end{subfigure}
    \caption{Computational experiments comparing Monte-Carlo methods and First-order approximation method}
\end{figure}

\section{Application to North American Natural Gas Market} 
\label{sec:NANGAM}
\begin{figure}
	\includegraphics[width=0.57\textwidth]{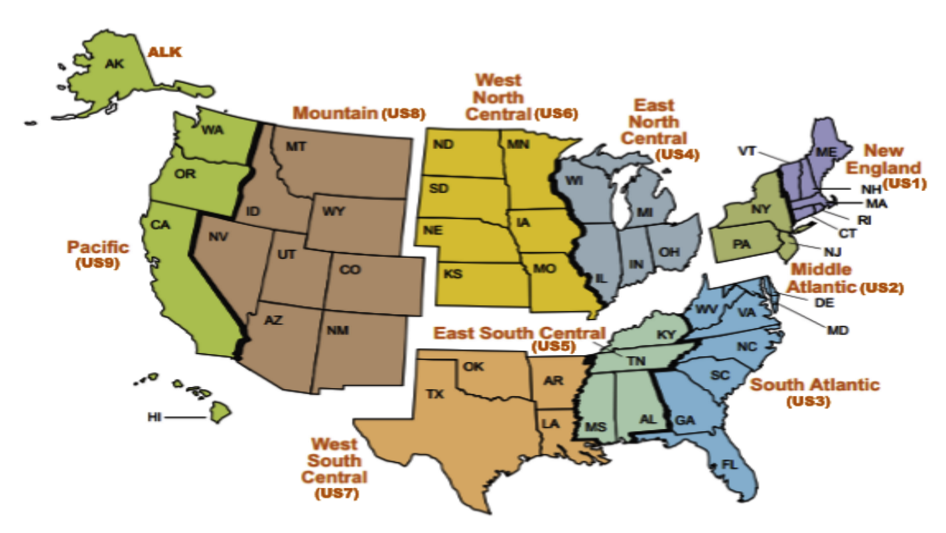}
	\includegraphics[width=0.4\textwidth]{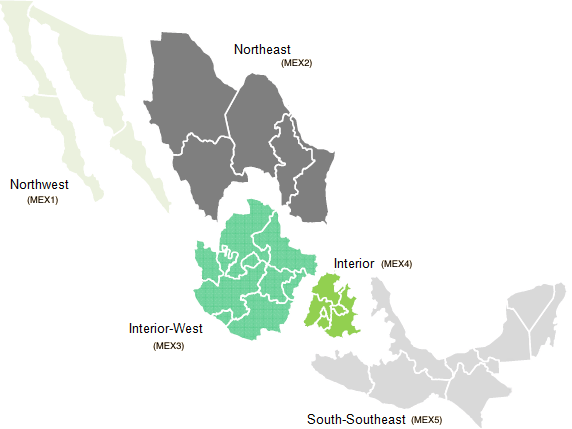}
	\caption{Regional disaggregation of United States and Mexico. Source: \cite{outlook2015us} and U.S. Energy Information Administration \url{http://www.eia.gov/todayinenergy/detail.php?id=16471}}
	\label{fig:Division}
\end{figure}
In Mexico, motivation to move from coal to cleaner energy sources creates an increasing trend in natural gas consumption, {particularly in the power sector}.{ Technology change, including fracking, has made natural gas available at a low cost. This resulted in increased production and higher proven reserves in the U.S. Therefore, the US is expected to become a net exporter of Natural Gas (increasing pipelines exports to Mexico and LNG) during the next years \citep{MexIncAEO,Feijoo}.}  The North American Natural Gas Model (NANGAM) developed in \cite{Feijoo} analyzes the impacts of cross border trade with Mexico. NANGAM models the equilibrium under various scenarios by competitively maximizing the profits of suppliers and pipeline operators, and the utility of consumers, resulting in a complementarity problem. The model also uses the Golombek function \citep{golombek1995effects,huppmann2013endogenous} to model the increase in marginal cost of production when producing close to capacity. The formal description of the model is provided in Appendix B. 

{For the model in this paper,} which is motivated by NANGAM, we have disaggregated the United States into 9 census regions (US1-9) and Alaska \citep{outlook2015us}. Mexico is divided into 5 regions (MEX1-5). A map showing this regional disaggregation is shown in Fig. \ref{fig:Division}. Further Canada is divided into two zones, Canada East {(CAE)} and Canada West {(CAW)}. The model has 13 suppliers, 17 consumers, 17 nodes, and 7 time-steps. This amounts to 12,047 variables (primal and dual) and 2023 parameters. The gradient matrix of the complementarity function would contain $12,047^2$ elements and a Hessian matrix will have $12047^3$ elements which is more than 1700 trillion floating point variables. We need efficient methods to handle these large objects. We observe, however, that the dependence of each component of the complementarity function is limited to few variables, thus making the gradient matrix sparse. Efficient sparse matrix tools in scipy \citep{jones2015scipy} are used along with a python class we specially built to handle a sparse multi-dimensional array. The details of this class are given in Appendix C.

This model is calibrated to match the region-wise production and consumption data by adjusting the parameters of the demand curve, supply curve and the transportation cost. The source for the projected numbers are the same as the ones in Table 2 of \cite{Feijoo}. The parameters of the demand curve were chosen in such a way that an elasticity of 0.29 is maintained at the solution to be consistent with \cite{elasticity}. 
\subsection{Covariance Matrix Calibration} 
\label{sub:covariance_matrix_calibration}
We used the method developed in algorithm \ref{alg:Algoaprx} to understand the propagation of uncertainty in the model. The covariance for each parameter across years is obtained by fitting a Wiener process to the parameter value. This is chosen to mimic the Markovian and independent increment properties of market parameters. Thus we have for any parameter
\begin{align}
d\randa(t) \quad &= \quad d\mu_{\randa}(t) + \sigma_{\randa}dB(t)
\end{align}
where $\mu_{\randa}$ is calibrated, $\sigma_{\randa}$ is chosen to be $1\%$ of the average value of $\mu_{\randa}$ in the analyzed period and $B(t)$ is the standard Brownian motion. The diffusion parameter $\sigma_{\randa}$ is assumed to be independent of time. Additionally to understand the effect of greater uncertainty in US7, that accounts for about 40\% of the total production in the continent, the parameters of production cost are assumed to have 5 times the variance than in any other region.
\subsection{Results}
\label{sub:res}
The deterministic version of the problem is solved using the PATH algorithm \citep{Dirkse1995} by assuming a mean value for all random parameters. Following this, algorithm \ref{alg:Algoaprx} was applied and the $\T$ matrix defined in \eqref{eq:TheResult} is obtained by solving the linear system of equations using a Moore-Penrose pseudoinverse \citep{horn2012matrix}. In the following paragraph, we discuss some of the results obtained in this study. 

The heat map on  {Fig. \ref{fig:VarCovara}} shows the coefficient of variation (standard deviation divided by mean) in consumer price in each year caused by the uncertainty in parameters as mentioned in subsection \ref{sub:covariance_matrix_calibration}. We notice that this large uncertainty in production costs of US7 caused relatively small uncertainties in the consumer price. This is partially due to the availability of resources in US8 and CAW to compensate for the large uncertainty in US7. The fact that it is actually US8 and CAW that compensate for this uncertainty is known by looking at the covariance plot on  {Fig. \ref{fig:VarCovarb}} which shows large correlation between US7 and US8 and also between US7 and CAW.

\begin{figure}[t]
\centering
\begin{subfigure}[b]{0.45\textwidth}
	\includegraphics[width = \textwidth]{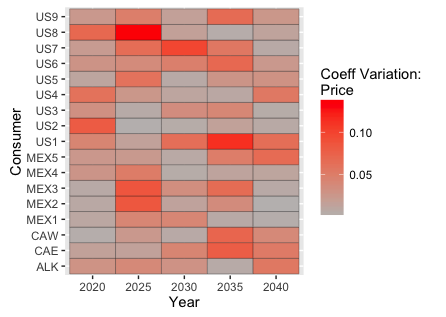}
	\caption{Coefficient of variation in Price}\label{fig:VarCovara}
\end{subfigure}
\begin{subfigure}[b]{0.45\textwidth}
	\includegraphics[width = \textwidth]{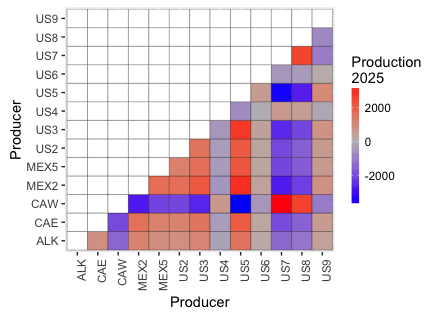}
	\caption{Covariance of Produced quantity}\label{fig:VarCovarb}
\end{subfigure}
	\caption{Covariance results}\label{fig:VarCovar}
\end{figure}

\begin{figure}[t]
\centering
\begin{subfigure}[b]{0.45\textwidth}
	\includegraphics[width = \textwidth]{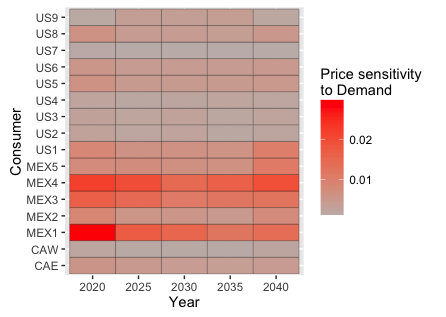}
	\caption{Price sensitivity to demand}\label{fig:tornadoa}
\end{subfigure}
\begin{subfigure}[b]{0.45\textwidth}
	\includegraphics[width = \textwidth]{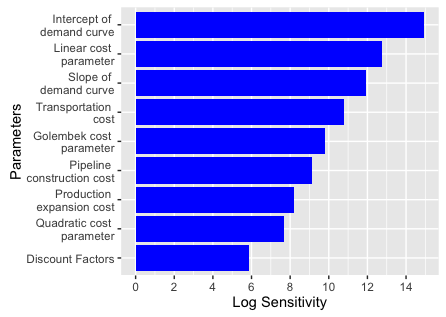} 
	\caption{Parameter sensitivity comparison}\label{fig:tornadob}
\end{subfigure}\caption{Sensitivity results}\label{fig:tornado}
\end{figure}

Fig. \ref{fig:tornado} shows the sensitivity of the solution to various input parameters. The graph on  {Fig. \ref{fig:tornadoa}} shows the sum total change in uncertainty in price for a 1\% fluctuation in the demand curve of  consumers. We notice that the price is particularly sensitive to changes in demand in Mexico.{ This reflects the increasing concern about growing exports (both LNG and pipeline) that are likely to result in higher consumer prices in the U.S. } We also note that fluctuations in demand at nodes where production facilities are not available (MEX1, MEX3, MEX4) cause greater uncertainty in price. This is because, for regions with a production facility in the same node, the production facility produces more to cater the demand at that node and there is little effect in the flows and in the prices at other nodes.{ This is also contingent to sufficient pipeline capacity. Larger changes in demand for regions with limited pupeline capacity (e.g. MEX1) may result in major changes in price.} However a perturbation to the demand at a node with no production unit causes the flows to alter to have its demand catered. This affects natural gas availability elsewhere and causes larger fluctuations in price. The tornado plot on  {Fig. \ref{fig:tornadob}} sorts the parameters in decreasing order of their effect on the uncertainty of the solution.

{The plot Fig. \ref{fig:tornadob} shows the total change of the equilibrium if a parameter (e.g., the demand intercept) is shifted by 1\% from its original value. Note that the results are plotted in logarithmic scale and are sorted in decreasing order. In general, our results suggest that parameters of consumers and producers play a major role on the equilibrium and hence a small perturbation have a large effect on the solution. In particular, the intercept and slope of the demand curve and the linear cost parameter are the three most significant parameters. Interestingly, the expansion cost parameters (pipeline as well as production expansion) have a lower effect on the solution equilibrium. As it was described on Fig \ref{fig:tornadoa}, uncertainties in demand significantly affect regions with no or low production capacities. Fig \ref{fig:tornadob} corroborates that changes to the demand affects the equilibrium the most. The results also indicate that the infrastructure expansion happens independently of changes in expansion cost. Natural gas prices paid by consumers account for cost expansions and transportation. Therefore, the level of expansion is then driven by changes on demand. Hence, if policy changes need to be implemented in order to, for instance, to increase economic activity or reduce carbon emissions, respectively subsidizes or taxes the downstream or upstream ends of market rather than the mid-stream players to have larger impacts. However, a policy maker who is interested in generating revenue without much impacts on the equilibrium should tax fuel transportation or infrastructure expansion for the greatest benefit.}

\section{Conclusion and Future work} 
\label{sec:conclusion_and_future_work}
In this paper, we developed a method to approximate the covariance of the output of a large-scale nonlinear complementarity problem with random input parameters using first-order approximation methods. We extended this method to general optimization problems with equality constraints. We then developed sensitivity metrics for each of the input parameters quantifying their contribution to the uncertainty in the output. We used these tools to understand the covariance in the equilibrium of the North American natural gas Market. The method gave insights into how production, consumption, pipeline flows, prices would vary due to large uncertainties. While the variances identified the regions that are affected the most, the covariance gave information about whether the quantity will increase or decrease due to perturbation in the input. We also obtained results on the sensitivity of price uncertainty to demand uncertainty in various nodes. We then quantified the contribution of each input parameter to the uncertainty in the output. This in turn, helps in identifying the regions that can have large impacts on equilibrium.

We note that the method is particularly useful for large-scale nonlinear complementarity problems with a large number of uncertain parameters, which make Monte-Carlo simulations intractable. It is robust in approximating the solution covariance for small uncertainty in the inputs. It is also good in quantifying the sensitivity of the output (and its variance) to the variance of input parameters. However since all the above are obtained as an approximation based on first-order metrics, there is a compromise in the accuracy if the variances of the input are large. The method works the best for problems involving a large number of decision variables and random parameters with small variance.

We foresee expanding this work by using progressively higher order terms of the Taylor series to capture the nonlinearities more efficiently. To ensure computational feasibility, this would typically require us to have stronger assumptions on the sparsity of the Hessian and the higher-order derivatives. This will also require analysis and stronger assumptions about higher-order moments of the random parameters.

\section{Acknowledgements} 
\label{sec:acknowledgements}
The model in this article is based in part on the multi-fuel energy equilibrium model MultiMod \cite{huppmann2014market}. The MultiMod was developed by Dr. Daniel Huppmann at DIW Berlin as part of the RESOURCES project, in collaboration with Dr. Ruud Egging (NTNU, Trondheim), Dr. Franziska Holz (DIW Berlin) and others (see \url{http://diw.de/multimod}). We are grateful to the original developers of MultiMod for sharing the mathematical implementation, which we further extended as part of this work.

The authors would also like to thank Dr. Donniell Fishkind, Department of Applied Mathematics and Statistics, Johns Hopkins University, Dr. Michael Ferris, Department of Computer Sciences, University of Wisconsin-Madison and the participants of TAI conference 2016, MOPTA 2016 and INFORMS 2016 for their valuable comments and discussions.

\sriAdd{The authors would also like to thank the two anonymous reviewers whose comments and suggestions improved this paper.}
\section{Proofs to certain lemmas and propositions} 
\label{sec:proofs_to_certain_lemmas_and_propositions}
\begin{proof}[Proof of Lemma 3] To show this, we first prove that every element in $\mathbb{K}'$ indeed is in $\mathbb{K}^*$. And then we prove for every element $\x \not\in \mathbb{K}'$, there exists some $\mathbf{v} \in \mathbb{K}$ such that $\mathbf{v}^T\x < 0$.\\
Consider an arbitrary $\x$ in $\mathbb{K}'$.
\begin{align}
\mathbf{v}^T\x \quad = \quad \sum_{i=1}^n \mathbf{v}_i\x_i \quad = \quad \sum_{i\in \mathcal{I}}\mathbf{v}_i\x_i + \sum_{i\not\in \mathcal{I}}\mathbf{v}_i\x_i \quad \geq \quad 0
\end{align}
where the final inequality follows from the fact that each term in the first summation is individually non-negative and each term in the second summation is 0. Thus we have $\mathbb{K}' \subseteq \mathbb{K}^*$.\\
Now to show the reverse containment, suppose there is $\x \in \mathbb{R}^n; \x\not\in \mathbb{K}'$. This means, we either have
\begin{enumerate}
	\item at least one index $j\in \mathcal{I}$ such that $\x_j <0$ or
	\item at least one index $j\not\in \mathcal{I}$ such that $\x_j \not = 0$
\end{enumerate}
Now,
\begin{align}
\mathbf{v}^T\x \quad = \quad \sum_{i=1}^n \mathbf{v}_i\x_i \quad = \quad \mathbf{v}_j\x_j + \sum_{i \not= j}\mathbf{v}_i\x_i 
\end{align}
In the first case, choose $\mathbf{v}$ such that $[\mathbf{v}]_i = 0$ for $i\not = j$ and $\mathbf{v}_j = 1$. Clearly $\mathbf{v}\in \mathbb{K}$ and for this choice of $\mathbf{v}$, the above sum is negative, showing $\x\not\in \mathbb{K}^*$. In the second case, choose $\mathbf{v}$ such that $[\mathbf{v}]_i = 0$ for $i\not = j$ and $\mathbf{v}_j = -\sgn(\x_j)$. Clearly $\mathbf{v}\in \mathbb{K}$ and for this choice of $\mathbf{v}$, the above sum is negative, showing $\x\not\in \mathbb{K}^*$. Thus we show $\left( \mathbb{K}' \right) ^c \subseteq \left( \mathbb{K}^* \right)^c$, which implies the reverse containment and completes the proof.
\end{proof}

\begin{proof}[Proof of Proposition 5]
	Since $\xstar\equiv\xstar(\randa)$ solves the problem, following from the requirement that $\cF(\xstar)\in \mathbb{K}^*$ and lemma 3, if $i\not\in \mathcal{I}$, $\cFi_i(\xstar;\randa) = 0$. \\
For $i\in \mathcal{I}$, $\xstar \in \mathbb{K} \Rightarrow \xstar_i \geq 0$ and $\cF(\xstar)\in \mathbb{K}^* \Rightarrow \cFi_i(\xstar)\geq 0$. Also from the requirement $\xstar^T\cF(\xstar) = 0$, one of the above two quantities should vanish for each $i\in \mathcal{I}$. But C-functions are precisely functions that vanish when both their arguments are non-negative and of them equal zero. So $\mer_i(\xstar,\cFi_i(\xstar)) =0$. \\
Thus each coordinate of $\Mer$ is individually zero, which makes $\f(\xstar)$ vanish, which is the smallest value $\f$ can take. Thus $\xstar$ is a global minimum of $\f$.
\end{proof}

\begin{proof}[Proof of Proposition 6]
	Since a solution exists for the NCP, we know by proposition 5 that the minimum value $\f$ can take is 0. Suppose we have $\xstar\in \mathbb{R}^n$ such that $\f(\xstar;\randa) = 0$. Since $\f$ is sum of squares, this can happen only if each of the individual terms are zero. This means for $i\not\in \mathcal{I}$, $\cFi_i(\xstar) = 0$. \\
Now since $\mer_i(\x_i,\cFi_i(\xstar)) = 0$ for $i\in \mathcal{I}$, we know $\cFi_i(\xstar) \geq 0$. This combined with the previous point implies $\cF(\xstar)\in \mathbb{K}^*$.\\
Also from the fact that $\mer_i(\xstar_i;\cFi_i(\xstar;\randa)) = 0$ for $i\in \mathcal{I}$, we know that $\xstar \in \mathbb{K}$. It also implies that $\xstar_i\cFi_i(\xstar) = 0$ for $i\in \mathcal{I}$. Thus
\begin{align}
\xstar^T\cF(\xstar;\randa) \quad &= \quad \sum_{i=1}^n \xstar_i\cFi_i\\
\quad &= \quad \sum_{i\in \mathcal{I}} \xstar_i\cFi_i + \sum_{i\not\in \mathcal{I}} \xstar_i\cFi_i\\
\quad &= \quad 0 + 0 \quad = \quad 0
\end{align}
This implies $\xstar(\randa) \perp \cF(\xstar;\randa)$ and $\xstar(\randa)$ solves the complementarity problem.
\end{proof}

\begin{proof}[Proof of Corollary 8]
	The set $O = \left\{ (a,b):\mer_{min}(a,b)=0 \right\} \setminus \left\{ (0,0) \right\}$ is the positive coordinate axes except the origin. Rewriting this C-function as
\begin{align}
\mer_{min}(a,b) \quad &= \quad \left \{
    \begin{array}{ll}
        a&\mbox{ if }a<b\\
        b&\mbox{ otherwise}
    \end{array}
\right .
\end{align}
we see that the second derivative of $\mer_{min}$ vanishes at $O$. Also we observe that
\begin{align}
\lim_{(a,b)\rightarrow(0,0)} \mer^2(a,b) \frac{\partial^2 \mer(a,b) }{\partial a\partial b} \quad &= \quad 0
\end{align}
since the second derivative is 0 everywhere except along the line $a=b$.
\end{proof}

\begin{proof}[Proof of Corollary 9]
	For $\mer = \mer_{FB}$, we have assumption 1 of proposition 7 satisfied by \cite{facchinei2007finite}. For assumption 2,
\begin{align}
\lim_{(a,b)\rightarrow(0,0)} \mer^2(a,b) \frac{\partial^2 \mer(a,b) }{\partial a\partial b} \quad &= \quad\lim_{(a,b)\rightarrow(0,0)}  \left( \sqrt{a^2+b^2}-a-b \right)^2\frac{ab}{\left( \sqrt{a^2+b^2} \right) ^3}\\
\quad &= \quad 0
\end{align}
Thus $\f$ is twice continuously differentiable at its zeros. The twice continuous differentiability elsewhere follows directly from the fact that $\mer_{FB}$ is twice continuously differentiable everywhere except at the origin. This ensures that all the terms in the derivative of the sum of squares exist and are finite.
\end{proof}

\begin{proof}[Proof of Proposition 14]
	By definition, for some admissible $d$,
\begin{align}
\f(\x+\delta e_d) \quad &= \quad \f(\x) + \delta Ge_d + o(\delta^2)\\
\Rightarrow [\f(\x+\delta e_d)]_i \quad &= \quad [\f(\x)]_i + \delta G_{id} + o(\delta^2)\label{eq:sens-lin}\\
\|\f(\x+\delta e_d)\|_2 \quad &\leq \quad \|\f(\x)\|_2 + \|\delta G_{.d}\|_2 + \|o(\delta^2)\|_2\\
\quad &= \quad \|\f(\x)\|_2 + \delta\sqrt{\left( \sum_{i=1}^n G_{id}^2 \right) } + o(\delta^2)\label{eq:sens-UB}
\end{align}
where $G_{.d}$ is the $d$-th column of $G$.  Also we have from \eqref{eq:sens-lin} for sufficiently small $\delta$,
\begin{align}
\|\f(\x+\delta e_d)\|_2 \quad &\geq \quad \|\f(\x)\|_2 - \|\delta G_{.d}\|_2 + \|o(\delta^2)\|_2\\
\quad &= \quad \|\f(\x)\|_2 - \delta\sqrt{\left( \sum_{i=1}^n G_{id}^2 \right) } + o(\delta^2)\label{eq:sens-LB}
\end{align}
\end{proof}

\begin{proof}[Proof of Corollary 17]
	For the problem to be well-defined, let $G \in \mathbb{R}^{n\times n}$ and $\randa \in \mathbb{R}^n$. This makes $\nabla^2_{\x\randa}\f(\x;\randa) \in \mathbb{R}^{n\times n}$.
\begin{align}
	\gradxf (\x;\randa) \quad &= \quad G\x + \randa(\randomxi)\\
	\hessxf(\x;\randa) \quad &= \quad G\\
	[\nabla^2_{\x\randa}\f(\x;\randa)]_{ij} \quad &= \quad I
\end{align}
Due to absence of terms dependent on $\x$ in the last two equations, we have an exact equation,
\begin{align}
G \Delta\x \quad &= \quad \perturb
\end{align}
Due to the exactness of the above equation, we have
\begin{align}
\T \quad &= \quad G^{-1}\\
\cov(\Delta\x) \quad &= \quad \T\cov(\perturb)\T^T
\end{align}
with no error.
\end{proof}


\section{Natural Gas Market - Complementarity Formulation} 
\label{sec:natural_gas_market_complementarity_formulation}
In this formulation, we assume we have a set of suppliers $\P$, consumers $\C$ and a pipeline operator. The players are located in a set of nodes $\N$, and some of them are connected by pipelines  $\A$.\\
Let also say that $\P_n\subseteq \P$, $\C_n\subseteq\C$ are located in node $ n\in \N$. Let $\A_n$ be the pipelines connected to node $n$. The symbols used here are explained in Table~\ref{tab:Sets}, \ref{tab:Symbols} and \ref{tab:Param}. Most of the analysis closely follow \cite{Feijoo} and \cite{Egging2016}. Random parameters are denoted by an $(\randomxi)$ beside them. The implementation of this problem is made available in \code{https://github.com/ssriram1992/Stoch\_Aprx\_cov}.

\begin{figure}
	\includegraphics[width=0.57\textwidth]{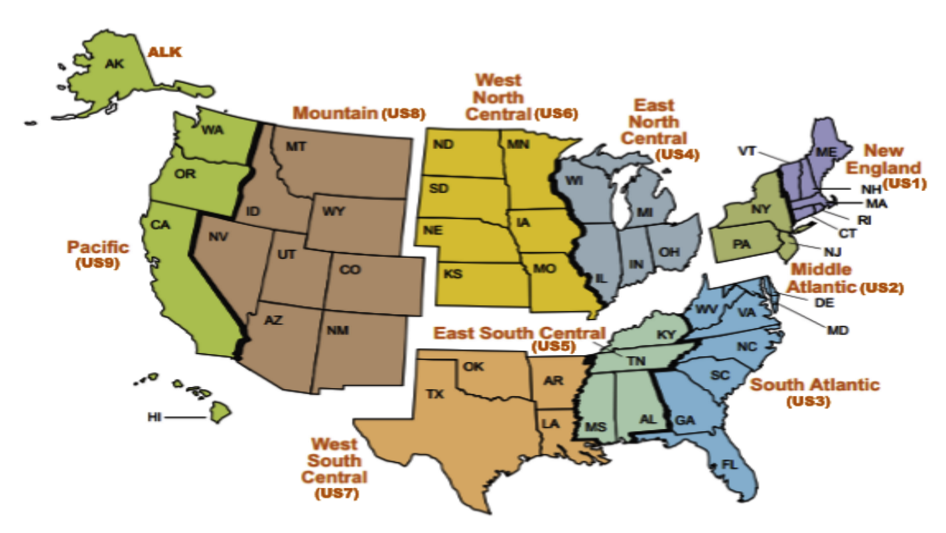}
	\includegraphics[width=0.4\textwidth]{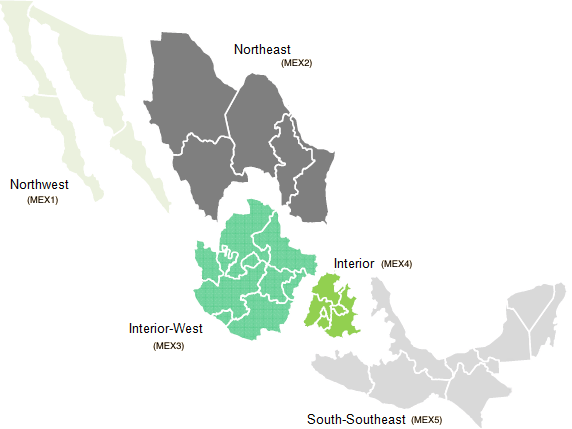}
	\caption{Regional disaggregation of United States and Mexico. Source: \cite{outlook2015us} and U.S. Energy Information Administration \url{http://www.eia.gov/todayinenergy/detail.php?id=16471}}
	\label{fig:Division}
\end{figure}
\subsection{Producer's problem} 
\begin{align}
 \Max \sum_{\Y}\df &\left\lbrace \sum_{\C} \QpcyC\piC -\costP\left(\QpyP,\CapP\right)\nonumber\right.\\
&\left.\quad -\piXP\Xp -\sum_{\Ano}\piA\Qpay \right \rbrace \label{eq:Prod}
\end{align}
subject to
\begin{subequations}
	\label{eq:ProdCon}
\begin{align}
\QpcyC,\QpyP,\Qpay \quad &\geq \quad 0 \nonumber\\
\Xp,\CapP \quad &\geq \quad 0 \nonumber\\
\QpyP \quad &\leq \quad \avl\CapP& \left( \dab \right) \\
\CapP \quad &= \quad \Qpo+\sum_{i=1}^y\Xp[i]& \left( \dac \right) \\
\sum_{\C_n}\QpcyC + \sum_{\Ano}\Qpay \quad &= \quad \QpyP(1-\lossP) \nonumber\\
\quad & \quad  \qquad+\sum_{\Ani}\Qpay(1-\lossA)& \left( \dad \right)
\end{align}
\end{subequations}
where
\begin{align}
	\costP(.) \quad &= \quad (\l+\g)\QpyP + \q\QpyP^2  \nonumber\\
	\quad & \quad \qquad +\g(\CapP-\QpyP)\log \left( 1- \frac{\QpyP}{\CapP} \right)
\end{align}

\begin{table}
\caption{Sets}\label{tab:Sets}
\small
\begin{tabular}{lcl|cl}
\hline
& \textbf{Set}& \textbf{Explanation}& \textbf{Set}& \textbf{Explanation}\\
\hline
&$\P$& Set of suppliers&$\A$& Set of pipeline connections(arcs)\\
&$\C$& Set of consumers&$\Ano$& Set of arcs from node $n$ on which natural gas flows out\\
&$\N$& Set of nodes &$\Ani$& Set of arcs from node $n$ on which natural gas flows in\\
&$\Y$& Set of periods\\
\hline
\end{tabular}
\end{table}

\begin{table}
\caption{Symbols - Variables}\label{tab:Symbols}
\small
\begin{tabular}{lcl}
\hline
& \textbf{Symbol}& \textbf{Explanation}\\
\hline
\textbf{Quantities}
&$\QpcyC$& Quantity produced by $p$ in $n$ to send to $c$ in year $y$\\
&$\QpyP$& Total quantity produced by $p$ in year $y$\\
&$\Qpay$& Total quantity $p$ choses to send by arc $a$ in year $y$\\
&$\Qa$& Total quantity sent by $a$ during year $y$\\
\textbf{Prices}&$\piC$& Unit price paid by consumer $C$ in year $Y$\\
&$\piA$& Unit price of sending natural gas through $a$ during year $y$\\
\textbf{Capacity}&$\Xp$& Production expansion in year $y$ for supplier $p$\\
&$\Xa$& Transportation capacity expansion in year $y$ for arc $a$\\
&$\CapP$& Production capacity for supplier $p$ in year $y$\\
&$\CapA$& Transportation capacity for arc $a$ in year $y$\\
\hline
\end{tabular}
\end{table}

\begin{table}
\caption{Symbols - Parameters}\label{tab:Param}
\begin{center}
\begin{tabular}{lcl}
\hline
& \textbf{Symbol}& \textbf{Explanation}\\
\hline
\textbf{Quantities}&$\Qpo$& Initial capacity of production for supplier $p$\\
&$\Qao$& Initial capacity of transportation for pipeline $a$\\
\textbf{Prices}&$\piXP$& Price of capacity expansion for supplier $p$\\
&$\piXA$& Price of capacity expansion for transportation arc $a$\\
\textbf{Losses}&$\lossP$& Percentage loss in production by supplier $p$ in year $y$\\
&$\lossA$& Percentage loss in transportation via arc $a$ in year $y$\\
&$\avl$& Availability fraction of the production capacity \\
\textbf{Consumer}&$\DemI$& Intercept of the demand curve for consumer $c$ in year $y$\\
&$\DemS$& Slope of the demand curve for consumer $c$ in year $y$\\
&$\df$& Discount Factor for year $y$\\
\hline
\end{tabular}
\end{center}
\end{table}

\subsection{Pipeline operator's problem} 
\begin{align}
\Max\sum_{\Y}&\df \left\{ \sum_{\A} \Qa \left( \piA-\costA \right) - \piXA\Xa  \right\}\label{eq:pipe}
\end{align}
{subject to }
\begin{subequations}
	\label{eq:pipeCon}
\begin{align}
\Qa,\Xa,\CapA \quad &\geq \quad 0 \nonumber	\\
\Qa \quad &\leq \quad \CapA&\left( \dah \right)\\
\CapA \quad &= \quad \Qao + \sum_{i=1}^y \Xa[i]&\left( \dai \right)
\end{align}
\end{subequations}

\subsection{Consumer} 
\begin{align}
\piC \quad &= \quad \DemI + \DemS\sum_\P \QpcyC&\left( \piC \right)
\end{align}
It can be shown that the above said optimization problems are all convex with non-empty interior. Hence the Karush-Kuhn Tucker conditions (KKT conditions) are necessary and sufficient for optimality. The KKT conditions are presented below and they form the equations for the complementarity problem along with the constraints above.

\subsection{KKT to Producer's problem}
\begin{subequations}
	\begin{align}
		-\df\piC + \dad  \quad &\geq \quad 0&\left( \QpcyC \right)\\
		\df\piXP - \sum_{i=1}^y\dac[{i}]\quad &\geq \quad 0&\left( \Xp \right)\\
		\df\piA + \left( \mathbb{I}_{a\in \Ano}  - \mathbb{I}_{a\in\Ani} (1-\lossA)  \right) \dad \quad &\geq \quad 0&\left( \Qpay \right)\\
		\df\frac{\partial \costP}{\partial \QpyP} +\dab - \dad(1-\lossP) \quad &\geq \quad 0& \left( \QpyP \right)\\
		\df\frac{\partial \costP}{\partial \CapP} + \avl\dac -\dab\quad &\geq \quad 0& \left( \CapP \right)
\end{align}
\end{subequations}

\subsection{KKT to Pipeline operator's problem} 
\begin{subequations}
	\begin{align}
		-\df\piA + \costA + \dah\quad &\geq \quad 0& \left( \Qa \right)\\
		\df\piXA- \sum_{i=1}^y \dai[i] \quad &\geq \quad 0&\left( \Xa \right) \\
	\dai - \dah \quad &\geq \quad 0&\left( \CapA \right)
	\end{align}
\end{subequations}

\subsection{Market clearing condition} 
\begin{align}
\Qa \quad &= \quad \sum_{\P}\Qpay&(\piA)
\end{align}

\section{N-dimensional Sparse array implementation} 
\label{sec:n_dimensional_sparse_array_implementation}
A general purpose Python class has been implemented to handle a sparse \emph{ndarray} object. The class is a generalization of the scipy class \emph{coo\_matrix} which stores the array coordinates of each non-zero element in the array. We now describe the details of the implementation. A continuously updated version of the class can be found at \url{https://github.com/ssriram1992/ndsparse}.

\subsection{Initialization} 
\label{ndA:initialization}
The n-dimensional sparse array (\code{coo\_array}) can be initialized by any of the following methods.
\begin{itemize}
	\item \emph{A \code{tuple}}, which initializes the sparse array of the shape mentioned in the \code{tuple} and with zeros everywhere.
	\item \emph{A dense \code{ndarray}} which will be converted and stored as a \code{coo\_array}.
	\item \emph{A \code{matrix} of positions and a 1 dimensional \code{array} of values} where the matrix contains the positions of the non-zero elements and the vector containing the non-zero values of those positions. In this case the shape of the \code{coo\_array} would be the smallest \code{ndarray} that can store all the elements given. Optionally a \code{tuple} containing the shape of the \code{ndarray} can be given explicitly.
	\item \emph{Another \code{coo\_array}} whose copy is to be created.
\end{itemize}

\subsection{Methods} 
\label{sub:methods}
The following methods and attributes are available in the \code{coo\_array}.
\begin{itemize}
	\item \code{print(coo\_array)} will result in printing the location of each of the non-zero elements of the array and their values.
	\item \code{coo\_array.flush(tol = 1e-5)} will result in freeing the space used in storing any zero-elements or elements lesser than the tolerance, \code{tol}. Such numbers typically arise out arithmetic operations on \code{coo\_array} or poor initialization.
	\item \code{coo\_array.size()} returns the number of non-zero elements in the \code{coo\_array}.
	\item \code{coo\_array.shape} returns the shape of the underlying dense matrix.
	\item \code{coo\_array.add\_entry(posn,val)} and \code{coo\_array.set\_entry(posn,val)} both add a new non-zero element with the given  value at the given position. The difference however is that \code{set\_entry()} checks if a non-zero value already exists at the mentioned position, and if yes, overwrites it. This search makes \code{set\_entry()} slower compared to \code{add\_entry()} which assumes that the previous value or the position is zero. Thus \code{add\_entry()} could potentially cause duplicates and ambiguity, if an illegal input is given. However in case the input is ensured to be legal, \code{add\_entry()} is much faster.
	\item \code{coo\_array.get\_entry(posn)} returns the value at the given position.
	\item \code{coo\_array.swapaxes(axis1,axis2)} is a higher dimensional generalization of matrix transposes where the dimensions that have to swapped can be chosen.
	\item \code{coo\_array.remove\_duplicate\_at(posn,func=0)} checks if there are multiple values defined for a single position in the sparse array. If yes, they are replaced by a single entry containing the scalar valued defined by \code{func} or passes them to a function defined in \code{func} and stores the returned value. Passing a function for the argument \code{func} is incredibly useful in performing arithmetic operations on \code{coo\_array}.
	\item \code{coo\_array.todense()} returns a dense version of the \code{coo\_array}.
	\item \code{coo\_array.iterate()} returns an iterable over the non-zero positions and values in the \code{coo\_array}.
\end{itemize}
The above class is used extensively to handle high-dimensional sparse arrays resulting out of variables containing pipelines, viz., $\Qpay, \Qa, \Qpo, \Qao,  \piA,  \Xa, \dah, \dai$ and parameters with pipelines, viz., $\CapA, \piXA,\lossA, \costA$.



\section{Computational Experiments}
\subsection{Problem Setup}
\sriAdd{The single-node single-product oligopoly mentioned in Section 5.1 is used for the computational experiments. Experiments were done with varying the number of players and to avoid any advantage due to symmetry, a different value was used for the cost of production for each of the $n$ players. The values used for the computational study are given below.}
\begin{align}
a \quad&=\quad 500\\
b \quad&=\quad -0.5\\
\mathbb{E}[c_i] \quad&=\quad 100 + 3i&(i=1,\,2,\,\ldots,\,n)\\
Var(c_i) \quad&=\quad 1&(i=1,\,2,\,\ldots,\,n)
\end{align}
\begin{figure}[t]
    \centering
    \includegraphics[width=0.7\textwidth]{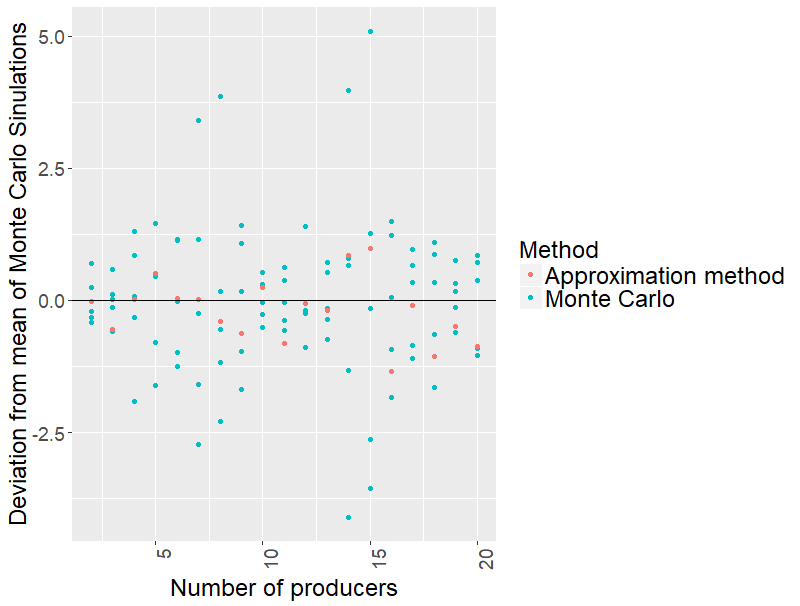}
    \caption{Error comparison}\label{fig:MCRace}
\end{figure}
\sriAdd{For the time tests, the comparison was between a Monte-Carlo simulation involving $0.1\times 2^n$ random samples and the approximation method. \\}
\sriAdd{For the error comparison, the Monte-Carlo simulation was done using $\max(100, 0.1\times2^n)$ samples. The process was repeated five times to show the relative differences between Monte-Carlo solutions purely due to the randomness in sampling. Having obtained the covariance matrix of the solution, both by Monte-Carlo simulation as well as the first-order approximation, we compute the trace of the covariance matrix. As said in Section 3, this trace corresponds to the invariant total uncertainty in the solution. Each cycle of the Monte-Carlo simulation gives one value of the said quantity. Each of these points is shown as a blue dot in the figure. The mean of all these values is considered as the zero line. The red dot corresponds to the value obtained by the approximation method. We hence demonstrate the relative accuracy of our method with respect to Monte-Carlo simulations. We also note that the test involves complementarity problems with exclusively strong complementarity terms which happens when the number of players is at most 15, and problems with both strong and weak complementarity terms, which happens when the number of players exceed 15. In either case, the error is comparable to that obtained via Monte-Carlo simulation.}

\sriAdd{For the purposes of this experiment, the first-order approximation method and the sampling process were implemented in Python 2.7 while the complementarity problems were solved using PATH algorithm accessed through Python-GAMS api. }
	
\section*{References}
\bibliographystyle{apa} 
\bibliography{All_ref,ref_only} 

\end{document}